\documentclass[10pt]{article}
\usepackage{amssymb,amsmath,amsthm,amscd}
\usepackage{epsfig, comment}
\usepackage{makeidx}
\usepackage[all]{xy}

\newtheorem{thm}{Theorem}
\newtheorem{cor}[thm]{Corollary}
\newtheorem{lem}[thm]{Lemma}
\newtheorem{prop}[thm]{Proposition}
\newtheorem{defn}[thm]{Definition}

\newtheorem{theorem-question}[thm]{Theorem-Question}

\newenvironment{exafont}{\begin{bf}}{\end{bf}}

\newenvironment{example}{\vspace{0.3cm}\par\noindent\refstepcounter{thm}\begin{exafont}Example \thethm\end{exafont}\hspace{\labelsep}}{\vspace{0.3cm}\par}

\newenvironment{remark}{\vspace{0.3cm}\par\noindent\refstepcounter{thm}\begin{exafont}Remark \thethm\end{exafont}\hspace{\labelsep}}{\vspace{0.3cm}\par}

\newcommand{\ssection}[1]{%
    \bigskip \begin{center} {\sc #1} \bigskip \end{center}}

\newcommand{\bA}{\underline{A}}
\newcommand{\bB}{\underline{B}}
\newcommand{\bC}{\underline{C}}

\newcommand{\cC}{\mathcal{C}}

\newcommand{\cX}{\mathcal{X}}

\newcommand{\ZZ}{\mathbb{Z}}
\newcommand{\eps}{\varepsilon}
\DeclareMathOperator{\Ob}{Ob}

\DeclareMathOperator{\Hom}{Hom}
\DeclareMathOperator{\grHom}{grHom}

\DeclareMathOperator{\add}{add}

\DeclareMathOperator{\Rad}{Rad}

\DeclareMathOperator{\ml}{-mod}

\DeclareMathOperator{\perf}{-perf}

\title{Homotopy, homology, and $GL_2$.}

\author{Vanessa Miemietz and Will Turner}
\date{23rd June 2008}

\begin{document}

\maketitle

\begin{abstract}
We define weak $2$-categories of finite dimensional algebras with bimodules,
along with collections of operators $\mathbb{O}_{(c,x)}$ on these $2$-categories.
We prove that special examples $\mathbb{O}_p$ of these operators control 
all homological aspects of the rational representation theory 
of the algebraic group $GL_2$, over a field of positive characteristic.
We prove that when $x$ is a Rickard tilting complex, 
the operators $\mathbb{O}_{(c,x)}$ honour derived equivalences, in a differential graded setting.
We give a number of representation theoretic corollaries, 
such as the existence of tight $\mathbb{Z}_+$-gradings on Schur algebras $S(2,r)$,
and the existence of braid group actions on the derived categories of blocks of these Schur algebras. 
\end{abstract}

\ssection{Introduction.}

Group theory has for a long time looked beyond itself for sources of intuition and understanding.  
Given that any topological space has a fundamental group, 
it is not surprising that one of those sources has been the theory of topological spaces.
Homology groups, originally defined as topological invariants, have developed into the techniques of homological algebra, which have proved a powerful tool \cite{BK}, 
and an intriguing source of questions in group representation theory \cite{Broue}.

We might expect more subtle topological invariants, such as homotopy groups, 
or invariants appearing in quantum field theory, 
to give rise to mathematical techniques, which also find application in group theory. 
One broad setting for homotopical algebra adopts a number of guises, 
which together form the subject of \emph{$n$-categories} \cite{Cheng}.
This setting also appears to be well adapted for the study of topological quantum field theories \cite{BaezDolan}.
The existence of gauge theories suggests that $n$-categories 
should indeed play a role in group representation theory, at least for small $n$ 
\cite{BFK}, \cite{ChuangRouquier}, \cite{Stroppel}.

In this paper, we present an example in which the formalism of $2$-category theory 
finds application in modular representation theory.
We define weak $2$-categories, and collections of combinatorial operators on these $2$-categories. 
Special examples of these operators control all homological aspects of the
rational representation theory of $GL_2$,  over a field of positive characteristic.

To be more precise, let $F$ be a field of characteristic $p>0$.
We give the collection $\mathcal{T}$ of $F$-algebras 
with a bimodule the structure of a weak $2$-category.
To every object $(c,x)$ of such a $2$-category which is positively graded, 
we associate a $2$-functor $\mathbb{O}_{(c,x)}$.
For a certain graded algebra $c_p$, of finite representation type, 
and a certain $c_p$-$c_p$-bimodule $x_p$, 
we thus obtain an operator $\mathbb{O}_p$.
Applying the operator $\mathbb{O}_p^n$ to the one-dimensional algebra $F$, 
with its trivial bimodule $F$,
we obtain an algebra, denoted $E_n$, with a bimodule. 
We prove that every block of rational $GL_2(F)$-modules is equivalent to the 
category of modules over an inverse limit $\lim_n E_n$ of these algebras.

The operators $\mathbb{O}_{(c,x)}$ are obtained by twisting a tensor product with a bimodule.
We perform a careful analysis of the effect of such twisted tensor products, 
in a differential graded setting.
For example, we prove if $x$ is a Rickard tilting complex, then the operator $\mathbb{O}_{(c,x)}$
respects derived equivalences.

Our theory links the combinatorially defined operator 
$\mathbb{O}_p$ with the representation theory of $GL_2(F)$, 
and has a number of corollaries. 
For example, the Schur algebra $S(2)$, 
which controls the category of polynomial representations of $GL_2(F)$, 
has a tight $\mathbb{Z}_+$-grading.
This follows since the operator $\mathbb{O}_p$  is tightly $\mathbb{Z}_+$-graded.
Furthermore, derived categories of certain blocks of the Schur algebra $S(2,r)$ admit an action of the braid group on $p$ braids. 
This follows from the existence of a braid group action on the derived category of $c_p$, 
established by Khovanov-Seidel and Rouquier-Zimmermann,
along with the analysis of a lift of the operator $\mathbb{O}_p$ to a differential graded setting.

We actually overturn a number of different algebras isomorphic to $E_n$.
One such, denoted $A_n$ in the paper, is a quiver algebra, modulo quadratic relations; 
a second, denoted $B_n$, is the basic algebra of a Schur algebra block;
a third, $C_n$, is obtained via the iteration of a certain trivial extension operator; 
a fourth, $D_n$, is a quotient of the basic algebra of a group algebra of finite special linear group.
We thus obtain a detailed understanding of the representation theory of all of these algebras.

The first section of this paper is combinatorial, 
and is concerned with establishing the isomorphism of various algebras,
applying work of Koshita, and Erdmann-Henke.
In the second section, we check out the $2$-category theory involved in our constructions.
In the third, we prove that the operators $\mathbb{O}_{(c,x)}$ respect derived equivalence in a differential graded setting.
In the fourth, we study the special operators $\mathbb{O}_p$.
In the fifth, we draw the various strands of argument together, 
and present applications to the representation theory of $GL_2(F)$.

\ssection{1. Four algebras.}\label{algebras}

Let $F$ be an algebraically closed field of characteristic $p >0$.

Let $n \geq 0$ be an integer.
Suppose $p \geq 3$.

We define four algebras, $A_n$, $B_n$, $C_n$, and $D_n$, eventually aiming to show they are all isomorphic:

\begin{enumerate}\label{algs}
\item \label{1} We define an algebra $A_n$ in terms of quivers and relations. Let $Q_n'$ be the quiver whose set of vertices are all tuples 
$(a_0, a_1, \dots, a_n) \in \{0, \dots, p-1\}^{n+1}$ in which, 
for a tuple $a= (a_0, a_1, \dots, a_n)$ and $1 \leq i \leq n$, 
there is an arrow $f_{i,\pm 1}^{a}$ from $a$ to 
$(a_0, a_1, \dots,a_{i-2}, p-2-a_{i-1}, a_i \pm 1, \dots, a_n)$ whenever $a_{i-1}$ is not $p-1$ and $a_i \pm 1 \in \{0, \dots, p-1\}$. 
Our desired quiver $Q_n$ is the connected component of $Q_n'$ containing $(0,0, \dots, 0)$.

In order to avoid unnecessary supercripts, we will write $f_{i,\pm 1}:= \sum f_{i,\pm 1}^{a}$ where the sum is over all vertices 
$a$ such that $f_{i,\pm 1}^{a}$ is an arrow in $Q_n$. 
Then $f_{i,\pm 1}^{a}=f_{i,\pm 1}a$ lies in the path algebra of $Q_n$.

Our first algebra $A_n$ is the path algebra of $Q_n$  modulo the ideal $I$ generated by the following relations:

\begin{itemize}
\item[(1)]$f_{i,\eps_1}f_{j,\eps_2}a-f_{j,\eps_2}f_{i,\eps_1}a$ if $i \neq j-1,j,j+1$
\item[(2)]$f_{i,\eps}f_{i,\eps} a$ for any $i, \eps$ and $a$ 
\item[(3)]$f_{i,\eps}f_{i,-\eps} a$ if $a_i=p-1$
\item[(4)]$f_{i,1}f_{i,-1} a -f_{i,-1}f_{i,1} a$ if $a_i>0$
\item[(5)]$f_{i,\eps_1}f_{i+1,\eps_2} a -f_{i+1,\eps_2}f_{i,-\eps_1}a$ if $a_{i}+ \eps_1, p-2-a_{i}- \eps_1$ are both in $\{0, \dots, p-1\}$,  and  $a_{i-1},a_i \neq p-1$.
\end{itemize}

\item \label{2} 
Let $S(2) = \oplus_{r \geq 0} S(2,r)$ denote the Schur algebra associated to $GL_2$, 
namely the graded dual of the coalgebra of regular functions
on the algebra of $2 \times 2$ matrices over $F$ \cite{Green}. 
Let $B_n$ be the basic algebra of a block of a Schur algebra $S(2,r)$ with $p^n$ simple modules, 
e.g. the principal block of $S(2,p^{n+1}-1)$.
\item \label{3} We now recall the construction of an algebra $\cC_p(A)$ where $A$ is an associative algebra, 
and $T$ is a self-dual bimodule over $A$, 
yielding another such algebra with a natural $\ZZ_+$-grading \cite{MT}. 
In order to construct $\cC_p(A)$ from $A$ and $T$, we first define an algebra $\mathcal B$ as the infinite matrix algebra
$$\left(   
\begin{array}{cccccc}
\ddots & {}_{i-2}T_{i-1} & 0 &  \cdots & \\
0& A_{i-1} & {}_{i-1}T_i& 0 & \cdots&\\
\cdots & 0 &  A_{i} & {}_{i}T_{i+1}& 0& \cdots\\
 &\cdots  &0&  A_{i+1} & {}_{i+1}T_{i+2}& 0\\
 & &&& A_{i+2}  &\ddots\\
&& &&&  \ddots \\
\end{array}
\right)$$
where all the $A_i$ and ${}_{i-1}T_i$ are isomorphic copies of $A$ and $T$ respectively. Denote by $\cC(A)$ the trivial extension of $\mathcal B $ by the $(\mathcal B,\mathcal B)$-bimodule $$\mathcal B^{(*)} = \bigoplus_{i \in \mathbb{Z}} \Hom_F(\mathcal B 1_{A_i},F).$$
Then  $\cC_p(A)$ is the subquotient of $\cC(A)$ given by 
$$\cC_p(A)= \eta_1^{p+1} \cC(A)\eta_1^{p+1}/\eta_1^{p+1} \cC(A)1_{A_{p+1}} \cC(A)\eta_1^{p+1} ,$$ 
where $\eta_1^{p+1}=\sum_{1 \leq i \leq p+1}1_{A_i}$.
The bimodule $$\cX_p(A) = \eta_1^{p} \mathcal{C}(A) \eta_2^{p+1}$$ admits the natural structure of a 
self-dual $\mathcal{C}_p(A)$-$\mathcal{C}_p(A)$-bimodule.
The operator $\mathcal{C}_p: (A,T) \mapsto (\mathcal{C}_p(A), \cX_p(A))$ 
acts on the collection of algebras with a self dual bimodule.
If we apply the $n$-fold composition operator $\mathcal{C}_p^n$ to the pair $(F,F)$, we obtain an algebra with a self-dual bimodule.
We call the algebra obtained in this way $C_n$. 

\item \label{4} Let $\widetilde{FSL(2,p^{m})}$ be the basic algebra of the group algebra of $SL(2,p^m)$ over an algebraically closed field of characteristic $p$. This algebra has a presentation in terms of quiver and relations by Koshita 
(\cite{Ko2} ,\cite{Kop}, see also \cite{Nebe} \cite{Nebe2}). 
We define our fourth algebra $D_n$ to be a quotient of $\widetilde{FSL(2,p^{n+2})}$. 
Let $e$ be the sum over all idempotents in Koshita's quiver 
which correspond to vertices not in the connected component of 
$(0,\dots,0)$. 
Let $f$ be the sum of idempotents corresponding to the vertices in the connected component of $(0,\dots,0)$, 
where the last entry in the $n+2$-tuple labelling the vertex is nonzero.
Then we define $D_n:=\widetilde{FSL(2,p^{n+2})}/\widetilde{FSL(2,p^{n+2})}(e+f)\widetilde{FSL(2,p^{n+2})}$.

\end{enumerate}
%


The main result of this section is

\begin{thm} \label{allfour}
All four algebras $A_n$, $B_n$, $C_n$, and $D_n$ are isomorphic.
\end{thm}

We prove the theorem via a sequence of lemmas and propositions.

\begin{remark} \label{inductiveconstruction}
The algebra $A_1$ is the well-known quasihereditary algebra
$$c_p = 
\xymatrix{F(\overset{1}{\bullet} \ar@/^/[r]^{\xi_1} &\overset{2}{\bullet}\ar@/^/[l]^{\eta_2} \ar@/^/[r]^{\xi_1} &\ar@/^/[l]^{\eta_2} \overset{3}{\bullet} &\cdots &\overset{p-1}{\bullet}\ar@/^/[r]^{\xi_{p-1}} & \ar@/^/[l]^{\eta_{p-1}} \overset{p}{\bullet})/I },$$
where
$I=(\xi_{i+1} \xi_{i}, \eta_{i} \eta_{i+1}, \xi_i \eta_i - \eta_{i+1} \xi_{i+1} , \xi_{p-1} \eta_{p-1})$,
for $1 \leq i \leq p-2$,
and in the tuple notation vertex $1$ corresponds to the tuple $(a_0,a_1)=(0,0)$, 
whilst vertex $2$ corresponds to $(a_0,a_1)=(p-2,1)$, vertex $3$ corresponds to $(a_0,a_1)=(0,2)$, etc.

We can inductively define the quiver $Q_n$ of $A_n$ as the quiver 
$$
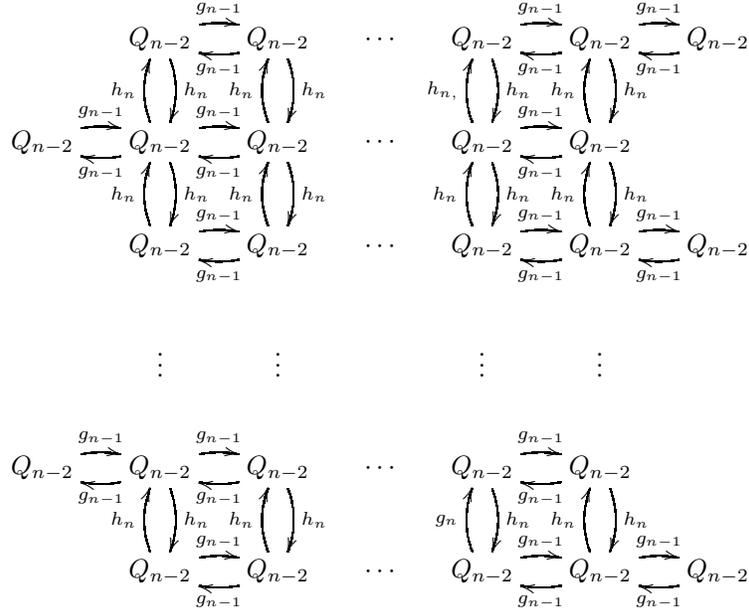
\xymatrix@C=15pt{       & Q_{n-2} \ar@/^/[d]^{h_{n}} \ar@/^/[r]^{g_{n-1}} & \ar@/^/[l]^{g_{n-1}} Q_{n-2} \ar@/^/[d]^{h_{n}}   &\cdots &Q_{n-2}\ar@/^/[d]^{h_{n}}\ar@/^/[r]^{g_{n-1}}&\ar@/^/[l]^{g_{n-1}} Q_{n-2}\ar@/^/[d]^{h_{n}}\ar@/^/[r]^{g_{n-1}}&
\ar@/^/[l]^{g_{n-1}}Q_{n-2}\\
Q_{n-2} \ar@/^/[r]^{g_{n-1}}& \ar@/^/[l]^{g_{n-1}}Q_{n-2}\ar@/^/[u]^{h_{n}}\ar@/^/[d]^{h_{n}}\ar@/^/[r]^{g_{n-1}} &
\ar@/^/[l]^{g_{n-1}}Q_{n-2}\ar@/^/[u]^{h_{n}}\ar@/^/[d]^{h_{n}} &\cdots &
Q_{n-2}\ar@/^/[u]^{h_{n,}}\ar@/^/[d]^{h_{n}}
\ar@/^/[r]^{g_{n-1}} & \ar@/^/[l]^{g_{n-1}}Q_{n-2}\ar@/^/[u]^{h_{n}}\ar@/^/[d]^{h_{n}}&\\
& Q_{n-2} \ar@/^/[u]^{h_{n}} \ar@/^/[r]^{g_{n-1}}  &\ar@/^/[l]^{g_{n-1}} Q_{n-2} \ar@/^/[u]^{h_{n}}   &\cdots & Q_{n-2}\ar@/^/[u]^{h_{n}} \ar@/^/[r]^{g_{n-1}}&\ar@/^/[l]^{g_{n-1}} Q_{n-2} \ar@/^/[u]^{h_{n}} \ar@/^/[r]^{g_{n-1}}&\ar@/^/[l]^{g_{n-1}} Q_{n-2}\\
&\vdots &\vdots&&\vdots&\vdots&\\
Q_{n-2} \ar@/^/[r]^{g_{n-1}}& \ar@/^/[l]^{g_{n-1}} Q_{n-2} \ar@/^/[d]^{h_{n}} \ar@/^/[r]^{g_{n-1}}& \ar@/^/[l]^{g_{n-1}} Q_{n-2} \ar@/^/[d]^{h_{n}} &\cdots & Q_{n-2} \ar@/^/[d]^{h_{n}} \ar@/^/[r]^{g_{n-1}} & \ar@/^/[l]^{g_{n-1}} Q_{n-2} \ar@/^/[d]^{h_{n}}&  \\
& Q_{n-2} \ar@/^/[u]^{h_{n}} \ar@/^/[r]^{g_{n-1}}  & \ar@/^/[l]^{g_{n-1}} Q_{n-2} \ar@/^/[u]^{h_{n}}  & 
\cdots &Q_{n-2} \ar@/^/[u]^{g_{n}}\ar@/^/[r]^{g_{n-1}}&
\ar@/^/[l]^{g_{n-1}} Q_{n-2} \ar@/^/[u]^{h_{n}}\ar@/^/[r]^{g_{n-1}}&
\ar@/^/[l]^{g_{n-1}} Q_{n-2} \\
}$$ 
where the above diagram has $p$ rows, each containing $p$ copies of $Q_{n-2}$;
where $h_n$ denotes a collection of arrows pointing from the vertices of one copy of $Q_{n-2}$ 
to the vertices of a second copy of $Q_{n-2}$, so that the correspondence defined by these arrows 
is the identity map on vertices of $Q_{n-2}$; 
where each row of the diagram is inductively identified with a copy $Q_{n-1}$; 
and where $g_n$ is identified with the collection of 
arrows pointing from one row to the next, therefore a union of $p-1$ copies of $h_{n-1}$.

The sum $f_{n,1}$ (resp. $f_{n,-1}$) 
of arrows in $Q_n$ is equal to the sum of all arrows which point downwards (resp. upwards) in the diagram,
whilst the sum $f_{n-1,1}$ (resp. $f_{n-1,-1}$)
is equal to the sum of arrows which point rightwards (resp. leftwards) in the diagram in odd rows, 
and leftwards (resp. rightwards) in even rows.

Relations $R_n$ for the algebra $A_n$ can also be defined inductively, as follows:

$\bullet $ Each copy of $Q_{n-2}$ in the diagram is subject to relations $R_{n-2}$,
defining a split subalgebra isomorphic to $A_{n-2}$.

$\bullet$ Each row $Q_{n-1}$ of the diagram is subject to relations $R_{n-1}$, 
defining a split subalgebra isomorphic to $A_{n-1}$.

$\bullet$ Each column of the diagram is subject to relations $R_1$,
defining a split subalgebra isomorphic to $A_{n-2} \otimes c_p$, 
where the $i^{th}$ vertex of $c_p$ corresponds to a copy of $Q_{n-2}$ in the $i^{th}$ row.

$\bullet$ All squares commute.

$\bullet$ $H_{n} x = x H_{n}$,
for every arrow $x$ in $Q_{n-2}$, and every collection $h_n$ in the diagram;
where $H_n$ denotes the sum of arrows in $h_n$;
and where $x$ is understood to lie in the subquiver at the source of $h_{n}$, 
on the left hand side of this relation, 
whilst $x$ is understood to lie in the subquiver 
at the target of $h_{n}$, on the right hand side of this relation.

\end{remark}

\begin{lem}\label{hatnohat}
$A_n \cong D_n$.
\end{lem}

\proof This is true since the quiver and relations defining $A_n$ are exactly obtained from Koshita's quiver and relations for $\widetilde{FSL(2,p^{n+2})}$, 
taking first the connected component of $(0,\dots,0)$, and then factoring out all vertices of the form 
$(a_0, a_1, \dots, a_n, a_{n+1})$ with $a_{n+1} > 0$, 
before identifying the vertex $(a_0, a_1, \dots, a_n, 0)$ with $(a_0, a_1, \dots, a_n)$.
\endproof

\begin{lem}\label{AontoB}
There exists an epimorphism $\phi: D_n \twoheadrightarrow B_n$.
\end{lem}

\proof
We identify an $n$-tuple $a=(a_0, a_1, \dots, a_{n}) \in \{0,1, \dots,p-1\}^n$ indexing 
irreducible modules for $SL(2,p^n)$ as in Koshita's notation \cite{Kop} with the natural number 
$ \lambda_a:=\sum_{0 \leq n-1}a_ip^i$ indexing a simple module for $S(2,r)$.
We first consider the first two algebras. Theorem 3.2 of an article of M. DeVisscher \cite{De} states that there is an epimorphism
$$\pi: \widetilde{FSL(2,p^{n+2})} \twoheadrightarrow \widetilde{S(2,p^{n+2}-1)} \oplus \widetilde{S(2,p^{n+2}-2)}$$ where for an algebra $A$, we denote its basic algebra by $\tilde A$.
This epimorphism is defined via epimorphisms of projective indecomposable modules. 
Hence a set of primitive orthogonal idempotents $\{e_{i}\} \subset
\widetilde{FSL(2,p^{n+2})}$ maps to a set of primitive orthogonal
idempotents $\{\tilde e_i\} \subset \widetilde{S(2,p^{n+2}-1)} \oplus
\widetilde{S(2,p^{n+2}-2)}.$ Only considering the quotient $B_{n+1}$ of the
Schur algebras, we get a surjection $$\tilde \pi: \Lambda_{n+1}:=
\widetilde{FSL(2,p^{n+2})} / \widetilde{FSL(2,p^{n+2})} e \widetilde{FSL(2,p^{n+2})} \twoheadrightarrow
B_{n+1}$$ where $e$ is the sum over all idempotents $e_i$ such that
$\tilde e_i \notin B_{n+1}$. 
In Koshita's presentation of $\widetilde{FSL(2,p^{n+2})}$ in terms of quiver and relations, the quotient  $\widetilde{FSL(2,p^{n+2})}/ \widetilde{FSL(2,p^{n+2})} e \widetilde{FSL(2,p^{n+2})}$ is obtained by taking the connected component of $(0, \dots,0)$.
Now we have a surjection of $B_{n+1}$ onto
$B_{n}$ given by factoring out $B_{n+1} \tilde f B_{n+1}$ where $\tilde f=
\sum_{\tilde e_i \in B_n, i \geq p^{n+1}} \tilde e_i =\pi (f)$.  Since the
image of $\Lambda_{n+1} f \Lambda_{n+1}$ (where $f= \sum_{e_i \in \Lambda_{n+1}, i \geq
  p^{n+1}} e_i$) under $\pi$ is contained in $B_{n+1} \tilde f B_{n+1}$ we
have an induced morphism between cokernels 
$$\phi: \Lambda_{n+1} / \Lambda_{n+1} f \Lambda_{n+1}  \twoheadrightarrow B_{n+1} / B_{n+1} \tilde f B_{n+1} \cong B_{n}.$$
Now, note that since $(p^{n+2} \geq) i \geq p^{n+1}$, 
the coefficient of $p^{n+1}$ in the $p$-adic expansion of $i$ is nonzero, 
and the idempotents $e$ and $f$ in this theorem are defined exactly as in the definition of $D_n$. 
Therefore $ \Lambda_{n+1} / \Lambda_{n+1} f \Lambda_{n+1} \cong D_n$ and $\phi$ is our desired epimorphism.
\endproof

We claim that $\phi$ is an isomorphism.
By Lemma \ref{hatnohat} this is equivalent to showing that $A_n \cong B_n$.

\begin{prop}\label{AB} $A_n \cong B_n$.
\end{prop}

\proof
We first note several things that are obvious from the relations for $A_n$.
\begin{itemize}
\item 
\emph{There exists a basis of $A_n$
such that every path in $Q_n$ is either equivalent to a element of this basis or zero in $A_n$.}
This is the case because relations in $A_n$ are quadratic 
and either take the form $a_1a_2$, or $a_1a_2 - a_1'a_2'$,
for some arrows $a_1,a_2,a_1',a_2'$. 
\item 
\emph{Any path of length $r$ in the quiver $Q_n$ defines defines an element of the path algebra which is equal 
to an element of the form $f_{i_1 \epsilon_1} \dots f_{i_r \epsilon_r} a$, where $a$ is the source vertex of
the path.} We will represent paths in this form whenever it is convenient to do so. 
\item \emph{If we fix $c_1, \dots c_r \in [0,p-1]$, and take the subalgebra generated by the vertices $(a_0, \dots, a_{n-r},c_r, \dots, c_1)$
and all arrows $f_{i,\eps}$ for $i \leq n-r$ between them,
then there is a split embedding of $A_{n-r}$ into $A_n$ which takes 
$(a_0, \dots, a_{n-r})$ to $(a_0, \dots, a_{n-r},c_r, \dots c_1)$.}
We denote the image of this embedding by $A_{n-r}^{c_1, \cdots, c_r}$. 
The existence of a splitting follows because the relations are graded with respect to an $n$-coloring: 
i.e there is a $\ZZ^n_+$-grading on $A_n$ which descends from a $\ZZ^n_+$-grading on the path algebra, 
giving $f_{i,\eps}$ degree $(0, \dots ,0, 1, 0, \dots, 0)$ where $1$ is in the $i$th position.
We will use both notations $P_{A_{n-r}^{c_1, \cdots, c_r}}(a_0, \dots, a_{n-r},c_r,\dots c_1)$ 
and  $P_{A_{n-r}^{c_1, \cdots, c_r}}(a_0, \dots, a_{n-r})$ 
for the projective indecomposable $A_{n-r}^{c_1, \cdots, c_r}$-module whose simple top is indexed by the vertex 
$(a_0, \dots, a_{n-r},c_r, \dots c_1)$.
We denote the subquiver of $Q_n$ which generates $A_{n-r}^{c_1, \cdots, c_r}$ by $Q_{A_{n-r}^{c_1, \cdots, c_r}}$.
\item 
\emph{The projective indecomposable $A_n$-module corresponding to vertex $a$ 
has simple composition factors which correspond to vertices whose last coordinate is 
either $a_{n}, a_{n} +1$ or $a_{n}-1$.} 
Indeed, any path in $Q_n$ from $Q_{A_{n-1}^{a_n}}$ to $Q_{A_{n-1}^{a_n \pm j}}$, for $j \geq 2$, 
contains at least $j$ factors $f_{n,\pm}$, 
which by relations (1),(4) and (5) can be reduced to a path 
containing an expression $f_{n,\pm}^2$, which is zero by relation (2).
\item \emph{Let $c=a_{n}$. Then $P_{A_n}(a_0, \dots, a_{n-1},c)$ is of the form 
\begin{equation}\label{A-proj}
\begin{array}{c}
 P\\ \;\;K^+ \; K^-\\  L 
\end{array}
\qquad
\begin{array}{c}
 P\\ \;\; K^+\\  L
\end{array}
\qquad
\begin{array}{c}
 P\\ \;\; K^-
\end{array} \end{equation}
for $1 \leq c \leq p-2$, $c=0$ and $c=p-1$ respectively.}
Here the picture $\begin{array}{c}
 P\\ \;\;K^+ \; K^-\\  L 
\end{array}$ denotes a module filtration with top quotient $P$, 
with submodule $L$ and with middle subquotient $K^+ \oplus K^-$. 
The other pictures should be interpreted similarly. 
The first two filtrations follow from relation (4) and the last one from relation (5), each applied in case $i=n$.
The top quotient $P$ has composition factors belonging to $A_{n-1}^{c} \ml$, 
and is isomorphic to the projective $P_{A_{n-1}^{c}}(a)$). 
The modules $K^\pm$ have composition factors belonging to $A_{n-1}^{c \pm 1} \ml$, 
whilst $L$ again has composition factors belonging to $A_{n-1}^{c} \ml$.
\end{itemize}

The projective indecomposable $B_n$-module $P_{B_n}(a)$  
is a quotient of the projective indecomposable $P_{A_n}(a)$, 
by Lemmas \ref{hatnohat} and \ref{AontoB}.
We now recall some  properties of $B_n$, from a previous article of ours \cite{MT},
and from papers of Erdmann and Henke \cite{He2}, \cite{EH}. 
\begin{itemize}
\item \emph{There is a quasi-hereditary quotient of $B_n$, 
$$B_{n-1}^{j}= \epsilon_j B_n \epsilon_j / \epsilon_j B_n \epsilon_{j+1} B_n\epsilon_j \cong B_{n-1},$$ 
where $\epsilon_j = \sum_{i> (j-1)p^{n}}\tilde e_i$, 
and the isomorphism to $B_{n-1}$ takes an idempotent $e_i$ to an idempotent $e_{(j-1)p^{n}+i}$ }
A proof of this can be found in previous papers (\cite{He2}, Theorem 5.1, \cite{MT}, Proposition 29).  
By iteration, 
we obtain quotients $B_{n-r}^{a_{n},a_{n-2}, \dots, a_{n-r+1} } \cong B_{n-r}$ of $B_n$. 
We remark that we do not yet know whether these quotients are split. 
As in the case of $A_n$, we use both notations $P_{B_{n-r}^{a_{n}, \cdots, a_{n-r+1}}}(a_0, \dots, a_{n})$ 
and  $P_{B_{n-r}^{a_{n}, \cdots, a_{n-r+1}}}(a_0, \dots, a_{n-r})$ for projective indecomposable 
$B_{n-r}^{a_{n}, \cdots, a_{n-r+1}}$-modules.
\item \emph{The projective module $P_{B_n}(a)$ has a filtration \cite{MT}
\begin{equation}\label{B-proj}
\begin{array}{c}
 P_{B_{n-1}^{c}}(a)\\ T_{B_{n-1}^{c+1}}(\tilde a)\;\; \; T_{B_{n-1}^{c-1}}(\tilde a)\\ P_{B_{n-1}^{c}}(a)^*
\end{array} \qquad \begin{array}{l}
P_{B_{n-1}^{c}}(a)\\ T_{B_{n-1}^{c+1}}(\tilde a)\\ P_{B_{n-1}^{c}}(a)^*
\end{array} \qquad  \begin{array}{l}
 P_{B_{n-1}^{c}}(a)\\  T_{B_{n-1}^{c-1}}(\tilde a)
\end{array} 
\end{equation}
for $1 \leq c \leq p-2$,  $c=0$ and $c=p-1$ respectively.}
Here $\tilde a= (p-1-a_i)_{(i=0, \dots, n-1)}$. Indeed, 
the middle composition factors in these filtrations are given as the trace of all modules  
$P_{B_n}(a_0, \dots, a_{n-1}, c\pm 1)$ in $P_{B_n}(a)$ (\cite{MT}, Proposition 29),
and isomorphic to the indecomposable tilting module 
$T_{B_{n-1}^{c \pm 1}}(\tilde a)$ indexed by vertex
$\tilde a= (p-1-a_i)_{(i=0, \dots, n-1)}$ (\cite{EH}, Proposition 25).
\item \emph{$T_{B_{n-1}^{c \pm 1}}(\tilde a) \cong P_{B_{k}^{c \pm 1,0, \dots,0}}(a_0, \dots, a_k)$, 
if $k$ is the last non-zero entry in $\tilde a$.} 
See\cite{EH}, Proposition 25. 
Another way of saying this is that, in the filtration (\ref{B-proj}) ,  
$P_{B_n}(a)$ has subquotients $ P_{B_{n-1}^{c}}(a)$, $P_{B_{k}^{c \pm 1,0, \dots,0}}(a_0, \dots, a_k)$ 
and (if $c \neq p-1$) $ P_{B_{n-1}^{c}}(a)*$, 
where $k$ is the largest number less than $n$ such that $a_k < p-1$. 
\end{itemize}

We now proceed by induction on $n$ to show that $A_n$ and $B_n$ are isomorphic. This is trivial for $n=0$, so assume that $A_{n-1}\cong B_{n-1}$.

Then, for fixed $a$, $P_{B_{n-1}^{c}}(a)\cong P_{A_{n-1}^{c}}(a)$, 
so comparing filtrations (\ref{A-proj}) and (\ref{B-proj})
we conclude that $K^\pm$ surjects onto $T_{B_{n-1}^{c \pm 1}}(\tilde a)$, 
whilst $L$ surjects onto $P_{B_{n-1}^{c}}(a)^*$. 
It remains to show that these surjections are isomorphisms. 
Equivalently, it remains to show that $K^\pm \cong P_{A_{k}^{c \pm 1,0, \dots,0}}(a_0, \dots, p-2-a_k)$, 
and $L \cong P_{A_{n-1}^{c}}(a)^*$, 
since by the inductive hypothesis $A_{k}^{c \pm 1,0, \dots,0} \cong B_{k}^{c \pm 1,0, \dots,0}$.
and $A_{n-1}^{c} \cong B_{n-1}^{c}$.

We will only treat the case $0 < c < p-1$. The  cases $c=0$ and $c=p-1$ are similar, but easier. 

More facts about $A_n$.

\begin{itemize}
\item \emph{For a path $\tilde p= f_{i_1,\eps_1}\cdots f_{i_r,\eps_r}a$ in $Q_n$ with $i_j<k$ for $j=1,\dots,r$, \begin{equation}\label{largefirst}f_{k,\eps}f_{i_1,\eps_1}\cdots f_{i_r,\eps_r}a = f_{i_1,\rho(\eps_1)}\cdots f_{i_r,\rho(\eps_r)}f_{k,\eps}a\end{equation} in $A_n$, where $\rho(\eps_j)=\eps_j$ if $i_j<k-1$ and
$\rho(\eps_j)=-\eps_j$ if $i_j=k-1$.} 
This is an immediate consequence of relations (1) and (5). 
We denote $f_{i_1,\rho(\eps_1)}\cdots f_{i_r,\rho(\eps_r)}$ by $(\tilde p)^\rho$

\item \emph{There is a unique shortest path $q$ in $Q_n$ from a vertex $a$ 
to the subquiver $Q_{A_{n-1}^{c \pm 1}}$, 
and any other such path can be rewritten as a concatenation of this path with a path inside $Q_{A_{n-1}^{c \pm 1}}$.} 
Indeed, if $a_{n-1} \neq p-1$, then $f_{n, \pm 1}a$ is a path with target
$a'=(a_0, \dots,p-2- a_{n-1}, c \pm1)$ in $Q_{A_{n-1}^{c \pm 1}}$, 
of length $1$, and this is obviously the shortest possible path.
If $a=(a_0,\dots, a_k,p-1,\dots, p-1,c)$, 
then we claim the shortest possible path to $Q_{A_{n-1}^{c \pm 1}}$ is 
$$q:=f_{n,\pm1}f_{n-1,-1}f_{n-2,-1}\cdots f_{k+1,-1}a,$$ 
and the target of this path is $a'=(a_0, \dots,p-2- a_{k},0, \dots,0, c \pm1)$. 
It can be easily checked that $q$ is a path from $a$ to $a'$. 
To see that this is the shortest path, consider the arrows $f_{i, \eps}^b$ in the quiver. 
Recall there is no arrow $f_{i,\eps}^b$ if $b_{i-1}=p-1$. 
Therefore, walking along any path from $a$ to a vertex with last coordinate different from $c$ 
will successively alter coordinates $a_{k+1}, \dots, a_{n-1}$ taking values $p-1$, 
beginnining with the coordinate $a_{k+1}$. 
So when we represent our path, all the elements 
$f_{n,\pm1},f_{n-1,-1},f_{n-2,-1},\dots, f_{k+1,-1}$ will appear, in this order;
in between two of these, say $f_{n-r,-1}$ and $f_{n-r-1,-1}$, 
we will see elements $f_{j,\eps}$ for $j \leq n-r-1$. 
But by the previous bullet point, we can rewrite this path as $p'q$, for some $p'$.
This shows that the path $q$ is the shortest possible path from $a$ 
to $Q_{A_{n-1}^{c \pm 1}}$, and any other such path is 
equal in $A_n$ to a concatenation of $q$ with a path $p'$ inside $Q_{A_{n-1}^{c \pm 1}}$.
\end{itemize}

%

We now proceed to show that $K^\pm \cong P_{A_{k}^{c \pm 1,0, \dots,0}}(a_0, \dots, a_k)$.

\begin{itemize}
\item \emph{$K^\pm$ has a simple top corresponding to vertex $a'$ and is hence a quotient of $P_{A_{n-1}^{c \pm 1}}(a')$. }
This is a module-theoretic version of the statement in the previous bullet point 
that any path from $a$ to $Q_{A_{n-1}^{c \pm 1}}$,
is equivalent to a concatenation of the path $q$ with a path inside $Q_{A_{n-1}^{c \pm 1}}$.

%
\item \emph{$K^\pm \cong P_{A_{k}^{c \pm 1,0, \dots,0}}(a_0, \dots, a_k)$.} 
To see this, we establish that $K^\pm$ is a quotient of $P_{A_{k}^{c \pm 1,0, \dots,0}}(a_0, \dots, a_k)$. 
This is sufficient since we have already established a surjection 
$K^\pm \twoheadrightarrow P_{A_{k}^{c \pm 1,0, \dots,0}}(a_0, \dots, a_k)$.

To see that $K^\pm$ is a quotient of $P_{A_{k}^{c \pm 1,0, \dots,0}}(a_0, \dots, a_k)$, 
we show that $f_{i,\eps}q = 0$, for $k+1 \leq i <n$.
Indeed, apply $f_{i,\eps}$ to $q$, for $k+1 \leq i <n$. 
As the target of $q$ has $0$ in the $i^{th}$ coordinate, the choice $\eps =-1$ will yield zero. 
On the other hand, the choice $\eps=1$ gives
\begin{equation}\begin{split}\label{blub}
f_{i,1}&f_{n,\pm1}f_{n-1,-1}f_{n-2,-1}\cdots f_{k+1,-1}a\\&\overset{(1)}{=}f_{n,\pm1}f_{n-1,-1}f_{n-2,-1}\cdots f_{i+2,-1}f_{i,1}f_{i+1,-1}f_{i,-1}\cdots f_{k+1,-1}a\\
&\overset{(5)}{=}f_{n,\pm1}f_{n-1,-1}f_{n-2,-1}\cdots f_{i+2,-1}f_{i+1,-1}f_{i,-1}f_{i,-1}\cdots f_{k+1,-1}a
\end{split}\end{equation}
but this contains a term $f_{i,-1}^2$ which is zero by relation (2), hence $f_{i,1}q$ also vanishes. (The numbers on the equality indicate the relations used.)
Therefore $f_{i,\eps}q$ is only non-zero for $i \leq k$, 
which means that $K^\pm$ factors over $A_{k}^{c \pm 1,0, \dots,0}$, 
and is therefore a quotient of $P_{A_{k}^{c \pm 1,0, \dots,0}}(a_0, \dots, a_k)$, as desired.
\end{itemize}
As $K^\pm \cong P_{A_{k}^{c \pm 1,0, \dots,0}}(a_0, \dots, a_k)$, 
we now know that the quotient of $P_{A_n}(a)$ by $L$ is isomorphic to the quotient of 
$P_{B_n}(a)$ by $P_{B_{n-1}^{c}}(a)^*$.

It remains to show that $L \cong P_{A_{n-1}^{c}}(a)^*$. Let us look at the composition factors.
\begin{itemize}
\item \emph{$L$ has a simple head corresponding to vertex $(a_0, \dots,p-2- a_{k},0, \dots,0, p-2, c)$, 
and is therefore a quotient of $P_{A_{n-1}^{c}}(a_0, \dots,p-2- a_{k},0, \dots,0, p-2, c)$.}
Indeed, the shortest path from $a$ to $Q_{A_{n-1}^{c \pm 1}}$ and back to $Q_{A_{n-1}^{c}}$ is  
$$q' = f_{n,\mp 1}f_{n,\pm 1}f_{n-1,-1}f_{n-2,-1}\cdots f_{k+1,-1}a,$$ 
since it is obtained by multiplying the shortest path $q$ from $a$ to $Q_{A_{n-1}^{c \pm 1}}$ by $f_{n,\mp 1}$, 
only increasing the length by one. 
The target of $q'$ is $(a_0, \dots,p-2- a_{k},0, \dots,0, p-2, c)$, 
and by the same argument as for $K^\pm$, $L$ has a simple head as claimed.
\item \emph{If $a_{n-1}\neq p-1$, then $L \cong P_{A_{n-1}^{c}}(a)$.}
Indeed, the path $q'$ given in the previous bullet point is just $f_{n,\mp 1}f_{n,\pm 1} a$, with target $a$, 
hence $L$ is a quotient of $P_{A_{n-1}^{c}}(a)$. 
But $P_{A_{n-1}^{c}}(a)$ is self-dual, as by the inductive hypothesis it is isomorphic to 
$P_{B_{n-1}^{c}}(a_0, \dots, a_{n-1})$, which (in case $a_{n-1} \neq p-1$) is self-dual, 
as is visible the filtrations for $B_{n-1}$ given by Equation \ref{B-proj}.  
We already know that $L$ surjects onto $P_{A_{n-1}^{c}}(a)^* \cong P_{A_{n-1}^{c}}(a)$, and so the claim follows.
\end{itemize}
From now on, we  assume $a_{n-1}=p-1$.
\begin{itemize}
\item \emph{We have a filtration of $P_{A_{n-1}^{c}}(a)^*$ with factors 
$$\begin{array}{c} P_{A_{k}^{(c,p-2, 0,\dots,0)}}(a_0,\dots, a_{k-1},p-2-a_k)
\\P_{A_{n-2}^{c,p-1}}(a)^*\end{array}.$$}
Indeed, by the inductive hypothesis $A_{n-1}\cong B_{n-1}$, 
and so $P_{A_{n-1}^{c}}(a)^*$ has a filtration with factors
$\begin{array}{c} T_{A_{n-2}^{c,p-2}}(\tilde a)
\\P_{A_{n-2}^{c,p-1}}(a)^*\end{array}$ 
analogous to the filtration given in Equation (\ref{B-proj}). 
But again $ T_{A_{n-2}^{c,p-2}}(\tilde a) \cong P_{A_{k}^{c,p-2, 0,\dots,0)}}(a_0,\dots, a_{k-1},p-2-a_k)$  
(\cite{EH}, Proposition 25), and by the inductive assumption that $A_{n-1}\cong B_{n-1}$. 
We therefore have a filtration as claimed. 
\item \emph{The maximal quotient of $L$ with composition factors in $A_{n-2}^{c,p-2}$ is 
isomorphic to $P_{A_{k}^{c,p-2, 0, \dots,0}}(a_0, \dots, p-2-a_k)$.}
A proof of this runs similarly to the proof that 
$K^\pm \cong P_{A_{k}^{c \pm 1,0, \dots,0}}(a_0, \dots, a_k)$, 
by noting that $f_{i,\eps}$ applied to $(a_0, \dots,p-2- a_{k},0, \dots,0, p-2, c)$ 
yields zero, for $k+1 \leq i \leq n-2$.
Indeed, for $k+1 \leq i \leq n-2$,
\begin{equation*}\begin{split}f_{i,1}f_{n,\mp 1}f_{n,\pm 1}f_{n-1,-1}f_{n-2,-1}\cdots f_{k+1,-1}a &= f_{i,\eps} f_{n,\mp 1}q\\&=f_{n,\mp 1}f_{i,\eps}q =0
\end{split}\end{equation*}
by Equation (\ref{blub}) and $f_{i,-1}(a_0, \dots,p-2- a_{k},0, \dots,0, p-2, c)=0$, 
because $a_i=0$.
\item \emph{Those composition factors of $L$ which are not in $A_{n-2}^{c,p-2} \ml$
are all in $A_{n-2}^{c,p-1} \ml$, and the smallest submodule $L'$ containing all these factors is a quotient of 
$P_{A_{n-2}^{c,p-1}}(a_0, \dots, p-2-a_k, 0,\dots,0,p-2,p-1,c)$. }
This follows since $L$ has a simple head 
$(a_0, \dots, p-2-a_k, 0,\dots,0,p-2,p-1,c)$, which is the target of $f_{n-1,1} f_{n,\mp 1}q$. 
Indeed, all composition factors of $L$ are in $A_{n-1}^{c} \ml$. 
We therefore have to consider those composition factors which are in $A_{n-2}^{c,j} \ml$, for some $j \neq p-2$. 
We obtain paths in $A_n$ corresponding to such factors 
by applying $f_{n-1,\pm1}$ to $\tilde p f_{n,\mp 1}q$, 
for any path $\tilde p$ in $Q_{A_{n-2}^{c,p-2}}$. 
But such a path is equal to $(\tilde p)^\rho f_{n-1,\pm1}f_{n,\mp 1}q$ in $A_n$, 
so we can equivalently apply $f_{n-1,\pm1}$ directly to $f_{n,\mp 1}q$. 
Exploiting the commutation relations for $A_n$,  and the fact that $f_{n-1,\pm1}^2=0$, 
we obtain the claim.
\item \emph{The smallest submodule of $L'$ containing all composition factors outside $A_{n-3}^{c,p-1,p-2} \ml$ 
only has composition factors in $A_{n-3}^{c,p-1,p-1}$, and is a quotient of the projective   
$P_{A_{n-2}^{c,p-1,p-1}}(a_0, \dots, p-2-a_k, 0,\dots,0,p-2,p-1,p-1,c)$.}
This follows since, by repeating the argument showing that the quotient of $L$ 
with composition factors in $A_{n-2}^{c,p-2} \ml$ is $P_{A_{k}^{c,p-2, 0, \dots,0}}(a_0, \dots, p-2-a_k)$, 
we see that the quotient of $L'$ with composition factors in 
$A_{n-3}^{c,p-1,p-2}$ is isomorphic to $P_{A_{k}^{c,p-1,p-2, 0, \dots,0}}(a_0, \dots, p-2-a_k)$. 
Then, as in the previous bullet point, the claim follows.

\item \emph{$L$ has a composition series with composition factors $P_0,P_1,...,P_k, \hat L$, 
where $P_1 \in A_{n-2}^{c,p-2} \ml$, where $P_1 \in A_{n-3}^{c,p-1} \ml$, 
where $P_2 \in A_{n-4}^{c,p-1,p-1} \ml$,..., 
where $\hat L$ is a quotient of $P_{A_{k}^{c,p-1, \dots ,p-1}}(a_0, \dots, a_k)$, and where $a_k \neq p-1$,
for some $k$.}
This follows by iterating the arguments of previous bullet points.

\item \emph{$\hat L$ is isomorphic to $P_{A_{k}^{c,p-1, \dots ,p-1}}(a_0, \dots, a_k)$.}
First note that the projective module $P_{A_{k}^{c,p-1, \dots p-1,p-1}}(a_0, \dots, a_k)$ is self-dual
since $a_k \neq p-1$. 
Furthermore, since $L$ surjects onto $P_{B_n-1^c}(a)^* \cong P_{A_{n-1}^c}(a)^*$,  
by restriction, $\hat L$ surjects onto 
$P_{A_{k}^{c,p-1, \dots p-1,p-1}}(a_0, \dots, a_k)^*$.
Hence, by self-duality we have surjections in either direction between the finite-dimensional spaces
$\hat L$ and $P_{A_{k}^{c,p-1, \dots ,p-1}}(a_0, \dots, a_k)$.
These surjections are therefore isomorphisms.
\end{itemize}

Thanks to our belt of bullets, the surjection from $L$ to $P_{B_n-1^c}(a)^*$ is an isomorphism,
from which it follows that the algebras $A_n$ and $B_n$ are indeed isomorphic, as required.
 


\endproof

We have now seen that the algebras $A_n, B_n$ and $D_n$ are all isomorphic. 
As a corollary, we obtain the following result (conjectured in \cite{MT}).

\begin{lem}
$B_n \cong C_n$.
\end{lem}

\proof
As already observed, $A_n$ has a $\ZZ_+^n$-grading, 
in which $f_{i,\eps}$ lies in degree $(0, \dots, 0,1,0,\dots 0)$, 
where the $1$ is in the $i$th component.
By Proposition \ref{AB}, $B_n$ is graded in the same way.
We established previously that the graded ring $grB_n$ of $B_n$, 
taken with respect to a certain filtration, is isomorphic to $C_n$ (\cite{MT}, Corollary 33). 
The $\mathbb{Z}_+^n$-grading here is compatible with this filtration, which means to say that $B_n \cong grB_n$.
Therefore $B_n \cong C_n$.
\endproof

Looking the lemmas of this section together, 
it is visible that we have now completed the proof of Theorem \ref{allfour}.

\begin{remark} Throughout this section, we have assumed that $p \geq 3$. 
Here we state how to treat the case $p=2$, which is slightly different, 
although analogous statements to those for $p \geq 3$ hold.

We define an algebra $A_n$ as follows: 
Let $Q_n'$ be the quiver with vertices $a=(a_0, \dots a_n)$, 
where all $a_i \in \{0,1\}$ and arrows $f_{\pm 1, i}^a$  (for $1 \leq i \leq n$) 
point from $a$ to $a=(a_0, \dots, a_i \pm 1, \dots, a_n)$, 
whenever $a_i \pm 1 \in \{0,1\}$. 
Let $Q_n$ be the connected component in $Q_n'$ containing $(0, \dots, 0)$ and 
note that the coordinate $a_0$ in this component is always zero, hence we can omit it.
$A_n$ is then given as the path algebra of $Q_n$ modulo the following relations:

\begin{itemize}
\item[(i)]$f_{i,\eps_1}f_{j,\eps_2}a-f_{j,\eps_2}f_{i,\eps_1}a$ if $i \neq j-1,j,j+1$ (and $ a_i +\eps_1,a_j +\eps_2 \in \{0, \dots, p-1\} , a_{i-1}=a_{j-1}=0$) 
\item[(ii)]$f_{i,1}f_{i,-1} a$ if $a_i=1$ (and $a_{i-1}=0)$
\item[(iii)]$f_{i+1,\eps}f_{i,-1} f_{i,1}a -f_{i,-1}f_{i,1}f_{i+1,\eps} a$ (if $a_i=a_{i-1}=0$)
\item[(iv)]$f_{i,1}f_{i+1,\eps}f_{i,-1} a$ if $a_{i}=1, a_{i-1}=0$.
\end{itemize}

Relation (i) corresponds with relation (1) from the case $p \geq 3$, 
whilst relation (ii) corresponds with relation (3) from the case $p \geq 3$. 
Relations (2) and (4) from the case $p \geq 3$ are superfluous in case $p=2$, 
and relations (iii) and (iv) given here replace the previous relation (5).

We denote by $B_n$ a block of the Schur algebra $S(2,r)$ with $2^n$ simple modules, eg. $S(2,2^n-2)$.

The operator $\mathcal{C}_2$, and the algebra $C_n$ are defined in exactly the same way 
as in the case $p \geq 3$. 

We define $D_n$ as in case $p \geq 3$, 
only this time using Koshita's quiver and relations from \cite{Ko2} for $p=2$.

The same results hold as in case $p \geq 3$, and we have
$$A_n \cong B_n \cong C_n \cong D_n.$$
The proofs are entirely analogous.
\end{remark}

\ssection{2. Homotopical constructions.}\label{2functor}

Here we define a weak $2$-category $\mathcal{T}$ 
of algebras with a bimodule, and define operators $\mathbb{O}_{(c,x)}$ on $\mathcal{T}$.
An introduction to weak $2$-categories can be found in a note by Leinster \cite{Leinster}.

We define a $2$-category $\mathcal{T}$ with the following data:

\begin{itemize}
\item a collection $\Ob \mathcal{T}$ whose elements are pairs $\bA:= (A,{}_A T_A)$,  where $A$ is a finite-dimensional algebra, and ${}_A T_A$ is an $(A,A)$-bimodule (these elements are the $0$-cells);
\item for each pair $(\bA,\bB)$ a category $\mathcal{T}(\bA,\bB)$ as follows:
\begin{itemize}
\item
objects ($1$-cells of $\mathcal{T}$) are pairs $(M, \phi_M)$ where $M={}_A M_B$ is an $(A,B)$-bimodule and $\phi_M: {}_A T_A \otimes_A M \rightarrow M \otimes_B {}_B T_B $ is an$(A,B)$-bimodule isomorphism; 
\item morphisms ($2$-cells of $\mathcal{T}$) $(M, \phi_M) \rightarrow (N, \phi_N)$ are $(A,B)$-bimodule morphisms $f:M \rightarrow N$ such that the diagram 
$$\xymatrix{ {}_A T_A \otimes_A M\ar[r]^{\phi_M} \ar[d]^{1 \otimes f}& M \otimes_B {}_B T_B  \ar[d]^{f \otimes 1}\\
{}_A T_A \otimes_A N
 \ar[r]^{\phi_N}& N \otimes_B {}_B T_B
}$$
commutes.
\end{itemize}
\item functors $$\mathcal{T}(\bA,\bB) \times \mathcal{T}(\bB,\bC) 
\rightarrow  \mathcal{T}((\bA,\bC)$$ taking pairs of $1$-cells 
$((M, \phi_M),(\tilde M, \phi_{\tilde M}))$ to 
$(M \otimes \tilde M, (\phi_M \otimes 1)\circ (1 \otimes \phi_{\tilde M}))$, 
and pairs of $2$-cells $(f,g)$ to $f \otimes g$; 
furthermore a $1$-cell $I_A$ which is just the bimodule ${}_A A_A$ together with the obvious isomorphism 
$\phi_A:  A\otimes_A {}_A T_A \rightarrow {}_A T_A \otimes_A A$.
\end{itemize}
Since tensor products of bimodules are associative, up to isomorphism, 
and the associativity isomorphisms obey the relevant pentagon axiom,
$\mathcal{T}$ indeed forms a weak $2$-category.

Now let $(A, T)$ be an object of $\mathcal{T}$ and let $(c,x)$ be another object of $\mathcal{T}$, 
such that $c = \oplus_{j \geq 0} c^{(j)}$ is a $\mathbb{Z}_+$-graded algebra, and 
$x = \oplus_{j \geq 0} x^{(j)}$ is a $\mathbb{Z}_+$-graded bimodule over $c$.

\begin{defn}\label{c(A,T)}
The algebra $c(A,T)$ 
is the vector space $c(A,T) := \underset{j \geq 0}{\bigoplus}c^{(j)} \otimes_F T^j$, 
where $c^{(j)}$ is the $j$th graded component of $c$, 
and the formal expression $T^j$ denotes the $j$-fold tensor product over $A$ of $T$ with itself. 
The multiplication is given by 
$$( c_1 \otimes t_1\cdots t_j )( c_2\otimes t_{j+1}\cdots t_{j+k} )= (c_1c_2\otimes t_1\cdots t_j\cdots t_{j+k}),$$ 
for $c_1\in c^{(j)},c_2 \in c^{(k)}$ and $t_1\cdots t_j= t_1 \otimes \cdots \otimes  t_j$ for $t_i \in T$. 
\end{defn}

Note that $c(A,T)$ is clearly an associative algebra, 
namely the subalgebra of homogeneous tensors inside $c\otimes_F A[T]$, 
where $A[T]$ denotes the tensor algebra of $T$ over $A$. 

We call $c(A,T)$ a \emph{twisted tensor product}, 
since it resembles a tensor product $c \otimes A$,
twisted by the bimodule $T$.

Analogously, we define the bimodule $x(A,T):= \underset{j \geq 0}{\bigoplus} x^{(j)}\otimes_F T^j$, 
where $x^{(j)}$ is the $j$th homogeneous component of $x$. 
The left and right $c(A,T)$-actions on $x(A,T)$ are given by 
$$( c_1\otimes t_1\cdots t_j )(x_1 \otimes t_{j+1}\cdots t_{j+k} )(c_2\otimes t_{j+k+1}\cdots t_{j+k+l})=$$
$$( c_1x_1c_2 \otimes t_1\cdots t_j\cdots t_{j+k+l} ),$$  
for $c_1 \in c^{(j)},c_2 \in  c^{(l)},x_1 \in x^{(k)}$ and $t_i \in T$. This oviously defines a bimodule action.

We define $(c,x) \ast (A,T):= (c(A,T), x(A,T))$. 
The product $\ast$, defined on graded objects of $\mathcal{T}$, is not associative.
We write $\mathbb{O}_{(c,x)}$ for the operator $(c,x) \ast -$, defined on objects of $\mathcal{T}$.

We now claim that this operator is natural on a categorical level, namely we have the following lemma.

\begin{lem}
$\mathbb{O}_{(c,x)}$ lifts to an endo-$2$-functor of $\mathcal{T}$.
\end{lem}

\proof
We have the following data:
\begin{itemize}
\item $\mathbb{O}_{(c,x)}$ defines a correspondence on $0$-cells, taking 
$\bA=(A,{}_A T_A)$ to $(c,x)\ast\bA:= (c(A,T),x(A,T))$ as defined above;
\item $(\mathbb{O}_{(c,x)})_{\bA,\bB }$ is a functor 
\begin{align*}
\mathcal{T}(\bA,\bB) &\rightarrow \mathcal{T}(\mathbb{O}_{(c,x)}(\bA),\mathbb{O}_{(c,x)}(\bB))\\
(M,\phi_M) &\mapsto (c(A,T) \otimes_A M, \mathbb{O}_{(c,x)}(\phi_M))\\
 (f:M \rightarrow N) &\mapsto (1 \otimes f: c(A,T)\otimes M \overset{1 \otimes f}{\rightarrow} c(A,T) \otimes N);
\end{align*}
\end{itemize}

We first need to check that $c(A,T) \otimes_A M$ is indeed an $(c(A,T), c(B,T))$-bimodule. 
However, due to the isomorphism $\phi_M$ we have and isomorphism $A[T]\otimes_A M \cong M\otimes_B B[T]$, 
hence $c\otimes_F A[T]\otimes_A M \cong M\otimes_B B[T]\otimes_F c$ and this is a graded isomorphism. 
Hence it restricts to $c(A,T) \otimes_A M \cong M \otimes_Bc(B,T)$.

Next we need to define 
$$\mathbb{O}_{(c,x)}(\phi_M): 
x(A,T) \otimes_{c(A,T)} (c(A,T) \otimes_A M)  \rightarrow (M \otimes_B  c(B,T))\otimes_{c(B,T)} x(B,T).$$ 
This is done in the same way. We take the isomorphism $A[T]\otimes_A M \rightarrow M\otimes_B B[T]$,
and tensor it with $x$ to obtain an isomorphism 
$x\otimes_F A[T]\otimes_A M \rightarrow M\otimes_B B[T]\otimes_F x$ which is graded, 
and hence restricts to the homogeneous part
$x(A,T) \otimes_{c(A,T)} (c(A,T) \otimes_A M)\cong x(A,T)  \otimes_A M  \rightarrow M \otimes_B   x(B,T)\cong (M \otimes_B  c(B,T))\otimes_{c(B,T)} x(B,T)$.

It remains to check that the diagram
$$\xymatrix{x(A,T) \otimes_{c(A,T)} (c(A,T) \otimes_A M) \ar[r]^{\mathbb{O}_{(c,x)}(\phi_M)} \ar[d]^{1 \otimes 1 \otimes f}& (M \otimes_B  c(B,T))\otimes_{c(B,T)} x(B,T)  \ar[d]^{f \otimes 1 \otimes 1}\\
 x(A,T) \otimes_{c(A,T)} (c(A,T) \otimes_A N)
 \ar[r]^{\phi_N}& (N \otimes_B  c(B,T))\otimes_{c(B,T)} x(B,T)
}$$
commutes.
This follows by the same argument as above from commutativity of the diagram
$$\xymatrix{A[T] \otimes_A M \ar[r]^{A[\phi_M]} \ar[d]^{1 \otimes f}& M \otimes_B B[T] \ar[d]^{f \otimes 1}\\
A[T]  \otimes_A N
 \ar[r]^{\phi_N}& N\otimes_B B[T],
}$$
which follows directly from $\phi_N \circ (1\otimes f)=(f \otimes 1)\circ \phi_M$.

\begin{itemize}
\item For $\bA, \bB, \bC \in \mathcal{T}$, we have two functors from the category
$\mathcal{T}(\bA, \bB) \times \mathcal{T}(\bB, \bC)$ to
$\mathcal{T}(\mathbb{O}_{(c,x)}(\bA), \mathbb{O}_{(c,x)}(\bC))$,
obtained by applying the functors $\mathbb{O}_{(c,x)}$, and composition, in different orders.
There is a natural transformation $\eta_{\bA \bB \bC}$ between these functors.
\end{itemize}
Indeed, we define
$$\eta_{\bA \bB \bC}(M,N): 
c(A,T) \otimes_A M \otimes_{c(B,T)} c(B,T) \otimes_B N \rightarrow c(A,T) \otimes_A M \otimes_B \otimes N,$$
to be the natural isomorphism obtained by contracting along the isomorphism $M \otimes_{c(B,T)} c(B,T) \cong M$,
for $(M, \phi) \in \mathcal{T}(\bA, \bB)$, and $(N, \psi) \in \mathcal{T}(\bB, \bC)$.

The natural transformation $\eta_{\bA \bB \bC}$ is obviously compatible with identity morphisms,
and therefore $\mathbb{O}_{(c,x)}$ is a $2$-functor, as required.
\endproof

\begin{example}
Let $c_p$ be the quasi-hereditary algebra with $p$ simple modules, 
defined in the first chapter (see Remark \ref{inductiveconstruction}). 
This algebra has various incarnations, such as $A_1$, $B_1$, $C_1$, and $D_1$.
In particular, applying the operator $\mathcal{C}_p$ to the trivial algebra $F$ with its trivial bimodule $F$,
we obtain the algebra $c_p$, along with a $c_p$-$c_p$-bimodule, which we denote $x_p$.

Looking at its definition via quiver and relations (see Remark \ref{inductiveconstruction}), 
we see that $c_p = \oplus_{i, j \in \mathbb{Z}} c_p^{(i), (j)}$ 
is $\mathbb{Z}^2_+$-graded, with arrows $\xi$ in degree $(1,0)$ and arrows $\eta$ in degree $(0,1)$. 
Summing these gradings, we obtain a $\mathbb{Z}_+$-grading on $c$.
We define the operator $\mathbb{O}_p$ on $\mathcal{T}$ to be the operator $\mathbb{O}_{(c_p,x_p)}$,
defined with respect to this grading.
Thus,
$$\mathbb{O}_p(A,T) = (c_p,x_p) \ast \bA = (c_p(\bA), x_p(\bA)).$$
\end{example}

\ssection{3. Twisted tensor products.}\label{construction}

Here, we define twisted tensor products in a differential graded setting.
We show that, under favourable conditions, taking a twisted tensor product respects derived equivalences.

We first need to introduce some standard notation. 
In an additive category with an object $x$, we denote by $\add x$ the closure under direct sums and summands. 
If $x$ is an object in a triangulated category, 
$x \perf$ denotes the closure under direct sums, summands, shifts and extensions. 

If $c,d$ are finite dimensional algebras, whose bounded derived categories are triangle equivalent, 
then there exists a Rickard tilting complex ${}_c x_d$, 
for which the pair of adjoint functors $(x \otimes_d^L -, RHom_c(x,-))$ induces the equivalence \cite{Rickard}.
Here, by a Rickard tilting complex, we mean a complex of $(c,d)$-bimodules, 
projective on both sides, such that the right multiplication morphism $d \rightarrow RHom_c(x,x)$
and the left multiplication morphism $c \rightarrow RHom_d(x,x)$ are quasi-isomorphisms. 

We work in this section with finite dimensional graded algebras and dg algebras.
We also work with finite dimensional graded dg algebras.
These are algebras $\alpha$ with a bigrading $\alpha = \oplus_{i,j \in \mathbb{Z}} \alpha^{(i),(j)}$,
equipped with a differential $d$ of degree $(0,1)$, which obeys the Leibnitz rule 
with respect to the $\mathbb{Z}$-grading obtained by forgetting the $i$-coordinate.
Therefore, with respect to the $\mathbb{Z}$-grading obtained by forgetting the $i$-coordinate,
$\alpha$ is a dg algebra, 
whilst with respect to the grading obtained by forgetting the $j$-coordinate,
$\alpha$ is an ordinary graded algebra. 
We call the grading obtained by forgetting the $i$-coordinate \emph{the homological grading},
whilst we call the grading obtained by forgetting the $j$-coordinate \emph{the standard grading}. 

We denote a shift by $k$ in the standard grading by $\langle k \rangle$.
We thus work with the convention that 
$(M \langle k \rangle )^{(i)} := M^{(i+k)}$ for a standardly graded object 
$M = \oplus_{i \in \mathbb{Z}} M^{(i)}$.

If $\alpha$ is graded dg algebra, whilst $m$ and $n$ are graded dg $\alpha$-modules, 
we denote by $\Hom_\alpha(m,n)$ the collection of all $\alpha$-module homomorphisms between $m$ and $n$,
and $\grHom_\alpha(m,n)$ the collection of $\alpha$-module homomorphisms which respect the standard grading.
Thus $\Hom_\alpha(m,n) = \oplus_{i \in \mathbb{Z}} \Hom^{(i)}_\alpha(m,n)$  
is a direct sum of chain complexes $\Hom_\alpha^{(i)}(m,n) = \grHom_\alpha(m,n \langle j \rangle)$.
We denote by $D_{dg}(\alpha)$ the derived category of all dg $\alpha$-modules,
where the dg structure is taken with respect to the homological grading.


We now discuss twisted tensor products in the world of dg-algebras.

Let $c$ and $d$ be finite dimensional $\mathbb{Z}_+$-graded algebras. 
Let ${}_c x_d$ be a positively graded complex of $(c,d)$-bimodules. 
Let also $a$ be a finite dimensional algebra, and ${}_a t_a$ a complex of $(a,a)$-bimodules.

\begin{defn}\label{c(a,t)}
We define a graded dg-algebra $c(a,t)$ as the vector space 
$c(a,t) := \underset{j \geq 0}{\bigoplus}c^{(j)} \otimes_F t^j$, 
where $c^{(j)}$ is the $j$th graded component of $c$, 
and the formal expression $t^j$ denotes the $j$-fold tensor product over $a$ of $t$ with itself. 
The multiplication is given by 
$$( c_1 \otimes t_1\cdots t_j )( c_2\otimes t_{j+1}\cdots t_{j+k} )= (c_1c_2\otimes t_1\cdots t_j\cdots t_{j+k}),$$ 
for $c_1\in c^{(j)},c_2 \in c^{(k)}$, 
and $t_1\cdots t_j= t_1 \otimes \cdots \otimes  t_j$ for $t_i \in t$. 
The differential is the total differential on $t^j$ within each standard homogeneous piece. 
\end{defn}

As before, this is clearly an associative algebra. To check that this is indeed a dg-algebra, 
consider the dg-algebra $c\otimes_F a[t]$, 
and consider $c(a,t)$ as the subalgebra which consists of homogeneous tensors. 
It then follows immediately that this is a dg-subalgebra since it is closed under multiplication and the differential. 

Note that we have two independent gradings here, 
the standard grading coming from $c$ and the homological grading coming from the complex $t$. 
So in fact $c(a,t)$ is a graded dg-algebra.

\begin{defn} \label{x(a,t)} We define the complex 
$x(a,t):= \underset{j \geq 0}{\bigoplus} x^{(j)}\otimes_F t^j$, 
where $x^{(j)}$ is the $j$th graded component of $x$ and as such a complex in its own right. 
The differential on $x(a,t)$ is the total differential on $t^j$ within each standard homogeneous piece. 
\end{defn}

The complex $x(a,t)$ admits a $c(a,t)$-action, 
given by the obvious multiplication 
$$( c_1\otimes t_1\cdots t_j )(x_1 \otimes t_{j+1}\cdots t_{j+k} )= ( c_1x_1 \otimes t_1\cdots t_j\cdots t_{j+k}),$$ 
for $c_1 \in c, x_1 \in x^{(k)}$ and $t_i \in t$. 
This is indeed a graded dg-algebra action, since it is a restriction of the 
graded dg algebra action of $c \otimes_F a[t]$ on $x \otimes_F a[t]$.

The algebra $d(a,t)$ acts similarly on the right of $x(a,t)$. 
By construction, left and right actions commute, 
making $x(a,t)$ into a graded $(c(a,t),d(a,t))$-dg-bimodule.

The main theorem of this section is the following:

\begin{thm} \label{derivedequivalence}
Suppose that $x$ and $t$ are Rickard tilting complexes.
Then the dg bimodule $x(a,t)$ induces a triangle equivalence 
$$\xymatrix{
D_{dg}(c(a,t)) \ar@/^/[r]^{RHom(x(a,t),-)} & \ar@/^/[l]^{x(a,t) \otimes^{L}_{d(a,t)} -} D_{dg}(d(a,t))
}$$
\end{thm}

Two finite dimensional algebras are derived equivalent if, and only if, 
there is a Rickard tilting complex inducing that equivalence \cite{Rickard}.
So roughly speaking, 
this theorem says that if $x$ and $t$ induce derived equivalences, so does $x(a,t)$.
In particular, if $_cx_c$ is a Rickard tilting complex, then $\mathbb{O}_{(c,x)}$ respects derived equivalences.

To prove the theorem, we apply B. Keller's Morita theory for differential graded algebras.
Indeed, by Lemma 3.10 of Keller's ICM article \cite{Ke}, to prove Theorem \ref{derivedequivalence},
it is sufficient to show that 
$_{c(a,t)}x(a,t) \in c(a,t) \perf$, that $_{c(a,t)}$ is a generator of $D_{dg}(c(a,t))$,
and that the natural morphism 
$$d(a,t) \rightarrow \Hom_{c(a,t)}(x(a,t), x(a,t))$$ is a quasi-isomorphism.
We do this in a sequence of steps.

\begin{lem}\label{preserveqim}
If two positively graded complexes $x$ and $y$ are quasi-isomorphic as graded complexes, 
then $x(a,t)$ and $y(a,t)$ are quasi-isomorphic complexes.
\end{lem}

\proof
Since we assume our quasi-isomorphism $\phi$ to be graded, 
we can write it as a sum of quasi-isomorphisms $\phi_i$ (of vector spaces), 
where $\phi_i$ is maps $x$ in standard degree $i$ to $y$ in standard degree $i$. 
Then we can define a map 
$\phi_i(a,t): x(a,t)^{(i)}\cong x^{(i)}\otimes t^i \rightarrow y(a,t)^{(i)}\cong y^{(i)}\otimes t^i$ 
to be $\phi_i \otimes id$. 
Since we're tensoring over $F$ this will remain a quasi-isomorphism, 
so $x(a,t)$ and $y(a,t)$ are quasi-isomorphic in every degree and hence altogether.
\endproof

We now compare the structure of $\Hom_{c(a,t)}(x(a,t),x(a,t))$ to the structure of $\Hom_c(x,x)$. 
As a $c$-module $x$ is the direct sum of projectives in possibly shifted degrees, 
i.e. $x = \underset{1\leq k \leq r}{\bigoplus} ce_{\lambda_k} \langle -i_k\rangle$ for $i_k >0$. 
The degree $j$ component $\Hom_c^{(j)}(x,x)$ -- 
where we take the grading with respect to the grading of $c$ and ignore that coming from the complex structure -- 
is then given by 
\begin{equation*}\begin{split}\Hom_c^{(j)}&(\underset{1\leq k \leq r}{\bigoplus} ce_{\lambda_k} \langle -i_k\rangle,\underset{1\leq l \leq r}{\bigoplus} ce_{\lambda_l} \langle -i_l\rangle) \\&= \grHom_c( \underset{1\leq k \leq r}{\bigoplus} ce_{\lambda_k} \langle -i_k\rangle,\underset{1\leq l \leq r}{\bigoplus} ce_{\lambda_l} \langle j-i_l  \rangle)\\
& \cong\underset{1\leq k \leq r}{\bigoplus}\underset{1\leq l \leq r}{\bigoplus} \grHom_c(  ce_{\lambda_k},ce_{\lambda_l} \langle j-i_l+i_k\rangle).\end{split}\end{equation*}
Since the $ ce_{\lambda_k}$ are projective, 
\begin{equation*}\begin{split} \underset{1\leq k \leq r}{\bigoplus}\underset{1\leq l \leq r}{\bigoplus} \grHom_c&(  ce_{\lambda_k},ce_{\lambda_l} \langle j-i_l+i_k \rangle)\\& \cong \underset{1\leq k \leq r}{\bigoplus}\underset{1\leq l \leq r}{\bigoplus} \grHom_{c^{(0)}}(  ce_{\lambda_k}^{(0)},(ce_{\lambda_l} \langle j-i_l+i_k \rangle)^{(0)})\\
&= \underset{1\leq k \leq r}{\bigoplus}\underset{1\leq l \leq r}{\bigoplus} \Hom_{c^{(0)}}(  ce_{\lambda_k}^{(0)},ce_{\lambda_l}).\end{split}\end{equation*}
Now we consider the degree $j$ component of $\Hom_{c(a,t)}(x(a,t),x(a,t))$.

\begin{lem} \label{homsapart}
$$\Hom^{(j)}_{c(a,t)}(x(a,t),x(a,t)) \cong 
\underset{1\leq k \leq r}{\bigoplus} \underset{1\leq l \leq r}{\bigoplus}\Hom_a(t^{i_k}, \Hom_{c^{(0)}}((c e_{\lambda_k})^{(0)},(c e_{\lambda_l}))\otimes t^{j+i_k})$$
\end{lem}

\proof
 Define $\tilde e := e \otimes 1_A$ for an idempotent $e \in c$.
Then, by definition, $x(a,t)$ decomposes as $x(a,t):= \underset{1\leq k \leq r}{\bigoplus} c(a,t) \tilde e_{\lambda_k} \langle -i_k \rangle \otimes t^{i_k}$ as a $c(a,t)$-module. Hence 
\begin{equation*}\begin{split}\Hom&^{(j)}_{c(a,t)}(x(a,t),x(a,t)) \\
&\cong \Hom^{(j)}_{c(a,t)}(\underset{1\leq k \leq r}{\bigoplus} c(a,t) \tilde e_{\lambda_k} \langle -i_k \rangle \otimes_a t^{i_k},\underset{1\leq l \leq r}{\bigoplus} c(a,t) \tilde e_{\lambda_l} \langle- i_l \rangle \otimes_a t^{i_l} )\\
&= \grHom_{c(a,t)}( \underset{1\leq k \leq r}{\bigoplus} c(a,t) \tilde e_{\lambda_k} \langle -i_k\rangle \otimes_a t^{i_k},\underset{1\leq l \leq r}{\bigoplus} c(a,t) \tilde e_{\lambda_l} \langle j-i_l  \rangle \otimes_a t^{i_l})\\
&\cong \grHom_{c(a,t)}( \underset{1\leq k \leq r}{\bigoplus} c(a,t) \tilde e_{\lambda_k} \otimes_a t^{i_k},\underset{1\leq l \leq r}{\bigoplus} c(a,t) \tilde e_{\lambda_l} \langle j-i_l+i_k \rangle \otimes_a t^{i_l})\\
&\cong \Hom_a(t^{i_k}, \grHom_{c(a,t)}(\underset{1\leq k \leq r}{\bigoplus} c(a,t) \tilde e_{\lambda_k},  \underset{1\leq l \leq r}{\bigoplus} c(a,t) \tilde e_{\lambda_l} \langle j-i_l+i_k \rangle \otimes_a t^{i_l}))\\
& \cong 
\underset{1\leq k \leq r}{\bigoplus} \underset{1\leq l \leq r}{\bigoplus}\Hom_a(t^{i_k}, \grHom_{c(a,t)}(c(a,t) \tilde e_{\lambda_k}, c(a,t) \tilde e_{\lambda_l} \langle j-i_l+i_k \rangle \otimes_a t^{i_l}))\\
&\cong \underset{1\leq k \leq r}{\bigoplus} \underset{1\leq l \leq r}{\bigoplus}\Hom_a(t^{i_k}, \grHom_{c(a,t)}(c(a,t) \tilde e_{\lambda_k}, c(a,t) \tilde e_{\lambda_l} \langle j-i_l+i_k \rangle ) \otimes_a t^{i_l})
\end{split}\end{equation*}
where the last isomorphism is by virtue of $c(a,t) \tilde e_{\lambda_k}$ being projective as a $c(a,t)$-module.
By the same argument, 
\begin{equation*}\begin{split} \grHom_{c(a,t)}&(  c(a,t) \tilde e_{\lambda_k},c(a,t) \tilde e_{\lambda_l} \langle j-i_l+i_k \rangle)\\& \cong \grHom_{c(a,t)^{(0)}}(  c(a,t) \tilde e_{\lambda_k}^{(0)},(c(a,t) \tilde e_{\lambda_l} \langle j-i_l+i_k \rangle)^{(0)})\\
&= \Hom_{c(a,t)^{(0)}}(  (c(a,t) \tilde e_{\lambda_k})^{(0)},(c(a,t) \tilde e_{\lambda_l}))\\& \cong 
\Hom_{c^{(0)} \otimes a}((c e_{\lambda_k})^{(0)} \otimes a,(c e_{\lambda_l})\otimes t^{j-i_l+i_k} )\\
& \cong \Hom_{c^{(0)}}((c e_{\lambda_k})^{(0)},(c e_{\lambda_l}))\otimes \Hom_a(a, t^{j-i_l+i_k})\\
& \cong \Hom_{c^{(0)}}((c e_{\lambda_k})^{(0)},(c e_{\lambda_l}))\otimes t^{j-i_l+i_k}.
\end{split}\end{equation*}
Note that $\grHom_{c(a,t)}(  c(a,t) \tilde e_{\lambda_k},c(a,t) \tilde e_{\lambda_l} \langle j-i_l+i_k \rangle)$ is only nonzero if $i_l-i_k-j >0$ since $x$ is concentrated in positive degrees, hence we only obtain positive tensor powers of $t$. 

Inserting the last equation back into  $\Hom^{(j)}_{c(a,t)}(x(a,t),x(a,t))$ yields
$$\Hom^{(j)}_{c(a,t)}(x(a,t),x(a,t)) \cong \underset{1\leq k \leq r}{\bigoplus} \underset{1\leq l \leq r}{\bigoplus}\Hom_a(t^{i_k}, \Hom_{c^{(0)}}((c e_{\lambda_k})^{(0)},(c e_{\lambda_l}))\otimes t^{j+i_k})$$
where the left $a$-module structure of the second argument is just that of $t^{j+i_k}$. This proves the lemma.
\endproof


\begin{lem}\label{cancelt}
The $(a,a)$-bimodule,
$\Hom_a(t^i,t^j)$ is quasi-isomorphic to the bimodule $\Hom_a(t^{i-j},a)$,  if $i >j$ and to $\Hom_a(a,t^{j-i})$ if $j \geq i$.
\end{lem}

\proof
Since $t$ is a Rickard tilting complex, 
we have a quasi-isomorphism $a \rightarrow \Hom_a(t,t)$. 
For $i=0$, or $j=0$, the claim is trivial, so assume inductively that it holds for $\Hom_a(t^{i-1},t^{j-1})$.
Then
\begin{equation*}\begin{split}
\Hom_a(t^i,t^j) &= \Hom_a(t \otimes_a t^{i-1},t \otimes_a t^{j-1})\\
&\cong \Hom_a(t^{i-1},\Hom_a(t,t \otimes_a t^{j-1}))\\
&\cong \Hom_a(t^{i-1},\Hom_a(t,t) \otimes_a t^{j-1})
\end{split}\end{equation*}

the last isomorphism depending on the fact that all modules occurring in $t$ are projective.
Again, since all modules are projectives on both sides, 
both functors $- \otimes_a t^{j-1}$ and $\Hom_a(t^{i-1},-)$ are exact, hence the quasi-isomorphism 
$a \rightarrow \Hom_a(t,t)$ remains a quasi-isomorphism under composition with both functors, 
and we obtain the desired quasi-isomorphism 
$\Hom_a(t^{i-1},a \otimes_a t^{j-1})\rightarrow \Hom_a(t^{i-1},\Hom_a(t,t) \otimes_a t^{j-1})$, 
which completes the induction.
\endproof

\begin{remark}\label{explicit} Note that the composite quasi-isomorphism $\Upsilon$ 
\begin{equation*}
\begin{split}\Hom_a(t^{i-1}, a\otimes_a t^{j-1}) &\rightarrow \Hom_a(t^{i-1}, \Hom_a(t,t) \otimes_a t^{j-1}) \\
&  \rightarrow  \Hom_a(t^{i},t^{j})\end{split}\end{equation*} 
is given by the explicit expression
\begin{equation*}
\begin{split}( t_1 \cdots t_{i-1} \mapsto 1 \otimes f(t_1 \cdots t_{i-1}))&\mapsto ( t_1 \cdots t_{i-1} \mapsto (t_i \mapsto t_i) \otimes f(t_1 \cdots t_{i-1}))\\
& \mapsto (t'_1 \cdots t'_{i}  \mapsto  t'_1 f(t_2' \cdots t'_i)).\end{split}\end{equation*} 
\end{remark}

\begin{lem}\label{perf}
$x(a,t) \in c(a,t)\perf$.
\end{lem}

\proof
We know that all summands of $x$ are projective $c$-modules, 
possibly shifted by a positive degree. 
Hence $x(a,t)$ is an iterated extension of projective $c(a,t)$-modules and projective $c(a,t)$-modules, 
tensored with further powers of $t$. 
However, as $t \in a\perf$, 
$P\otimes_a t$ for a projective $c(a,t)$-module $P$ is itself a complex of projective $c(a,t)$-modules, 
hence in $c(a,t) \perf$. 
So, $x(a,t)$ is a finite iterated extension of elements in $c(a,t) \perf$, hence in $c(a,t) \perf$ itself.
\endproof

\begin{lem}\label{gen}
$x(a,t)$ generates $D_{dg}(c(a,t))$.
\end{lem}

\proof
First note that $c(a,t)$ generates $D_{dg}(c(a,t))$, i.e. 
$$\textrm{ if }\Hom_{D_{dg}(c(a,t))}(c(a,t),M[i])=0 
\textrm{ for all }i \in \ZZ, \textrm{ then } M=0.$$
Note also that under the tilting equivalence between $c$ and $d$,  $c$ corresponds to $\Hom_c(x,c)$, and $x$ corresponds to $d$. Hence, since $\Hom_c(x,c)$ is in $d\perf$, $c$ is also in $x\perf$.

We now show that $c(a,t) \in x(a,t)\perf$. Consider an $x$-resolution $\tilde x$ of $c$ and let $j$ be the largest positive integer such that some $x_k \in \add x$ appears in $\tilde x$ as $x_k\langle -j\rangle$. Then $c \langle j \rangle$ has a resolution by objects of $\add x$ and non-negative shifts of those. Now by the same argument as in Lemma \ref{perf} we obtain that $c(a,t)\otimes_a t^j$ is in $x(a,t) \perf$ ($c(a,t)\otimes_a t^j$ is an iterated extension of some $x_i \in \add x(a,t)$ and some $x_i\otimes_a t^m$ for $x_i \in \add x(a,t)$ and $m>0$, the latter being in $x(a,t) \perf$, since $t$ is a finite complex of projectives). Now $t$ is a tilting complex, hence $a$ is in $t \perf$, say with a resolution $\tilde t$. Then $c(a,t) \otimes t^{j-1} \otimes_a \tilde t \in x(a,t) \perf$ is a resolution of $c(a,t) \otimes t^{j-1} \otimes_a a$ (since $c(a,t) \otimes t^{j-1}\otimes_a-$  is exact) and inductively we finally obtain that $c(a,t)$ is in $x(a,t) \perf$ as desired.

Now we can consider an extension $\tilde x: 0 \rightarrow x_d \rightarrow \cdots \rightarrow x_0 \rightarrow 0$ 
which is quasi-isomorphic to $c(a,t)$ in the $0$-component, and where all $x_i \in \add x(a,t)$. 
If $\Hom_{D_{dg}(c(a,t))}(x(a,t), Y[i])=0$ 
for all $i \in \ZZ$ for some $Y$, then by induction on the length of the iterated extension, 
$\Hom_{D_{dg}(c(a,t))}(\tilde x, Y[i])=0$. Indeed, there is a triangle 
$$x_d \longrightarrow \tilde x \longrightarrow \tilde x^{\leq d-1}\leadsto$$
where $\tilde x^{\leq d-1}= (0 \rightarrow x_{d-1} \rightarrow \cdots \rightarrow x_0 \rightarrow 0)$. 
From the long exact sequence we get by applying $\Hom_{D_{dg}(c(a,t))}(-, Y[i])$ 
and the inductive hypothesis that 
$\Hom_{D_{dg}(c(a,t))}(x_d, Y[i])=0 = \Hom_{D_{dg}(c(a,t))}(\tilde x^{\leq d-1}, Y[i])$, 
we obtain the claim. But since $\tilde x$ is quasi-isomorphic to $c(a,t)$ this implies that 
$\Hom_{D_{dg}(c(a,t))}(c(a,t), Y[i])=0$ for all $i \in \ZZ$, hence $Y=0$. 
This shows that $x(a,t)$ generates $D_{dg}(c(a,t))$.
\endproof

{\noindent \emph{Proof of Theorem \ref{derivedequivalence}.}}

First note that Lemma \ref{perf} and Lemma \ref{gen} show that we can apply \cite{Ke}, Theorem 3.10, to see that $c(a,t)$ is derived equivalent to  $\Hom_{c(a,t)}(x(a,t),x(a,t))$. 
Now we proceed to show that the natural map $d(a,t) \rightarrow \Hom_{c(a,t)}(x(a,t),x(a,t))$ coming
from the action of $d(a,t)$ on $x(a,t)$ is a quasi-isomorphism.

Define the vector space $\Hom_c(x,x)(a)$ as follows:

\begin{equation*}\Hom^{(j)}_c(x,x)(a):= \left\{ \begin{array}{ll}
\Hom^{(j)}_c(x,x) \otimes t^j & \hbox{if } j\geq 0\\
\Hom^{(j)}_c(x,x) \otimes \Hom(t^{-j},a) & \hbox{if } j < 0
\end{array}\right.. \end{equation*}

Note that the graded quasi-isomorphism $d \rightarrow \Hom_c(x,x)$ lifts to a quasi-isomorphism 
$d(a,t)\rightarrow \Hom_c(x,x)(a,t)$, since tensoring over $F$ with $t^j$ is an exact functor.

Now we wish to construct a quasi-isomorphism 
$$\beta: \Hom_c(x,x)(a,t) \rightarrow \Hom_{c(a,t)}(x(a,t),x(a,t)).$$ 

For an element $f_j \otimes t_1 \cdots t_j$, or $f_j \otimes h$, 
where $f_j \in \Hom_c^{(j)}(x,x)$, and $t_1 \cdots t_j \in t^j$ if $j \geq 0$, 
and where $h \in \Hom_a(t^{-j},a)$ if $j<0$, we define

$$\beta(f_j \otimes t_1 \cdots t_j)= 
(x_k \otimes t_1'\cdots t_k' \mapsto f_j(x_k) \otimes  t_1'\cdots t_k't_1 \cdots t_j),$$
$$\beta(f_j \otimes h)=$$
$$\left(x_k \otimes t_1'\cdots t_k'\mapsto
\left\{ \begin{array}{ll}
 0& \hbox{if } -j > k\\
 f_j(x_k) \otimes t_1'\cdots t_{k-(-j)}'h(t_{k-(-j)+1}'\dots t_k')  & \hbox{if } -j \leq k
\end{array}\right)\right.$$ 
The image under $\beta$ is clearly a collection of endomorphisms of $x(a,t)$ 
which commute with the action of $c(a,t)$ on the left 
(since the parts coming from $t^j$ or $\Hom_a(t^j,a)$ act from the right and $c(a,t)$ by multiplication on the left). 
Plugging the iterated version of the explicit map $\Upsilon$ given in Remark \ref{explicit} 
into the formula given in Lemma \ref{homsapart}, 
we obtain exactly our map $\beta$. 
Lemma \ref{cancelt} then implies that $\beta$ is a quasi-isomorphism.

Composing the quasi-isomorphism $d(a,t)\rightarrow \Hom_c(x,x)(a,t)$ with 
$\beta$ takes $d_l \otimes t_1\cdots t_l \in d(a,t)^{(l)}$ to 
$$(\textrm{right action of }d_l)\otimes t_1\cdots t_l \in \Hom^{(l)}_c(x,x)\otimes t^l,$$ 
then to 
$$(\textrm{right action of }d_l)\otimes (\textrm{right multiplication with }t_1\cdots t_l ),$$
an element of $\Hom_{c(a,t)}(x(a,t), x(a,t))$ 
This element is also the image of  $d_l \otimes t_1\cdots t_l \in d(a,t)^{(l)}$ under the canonical map 
$\gamma: d(a,t) \rightarrow \Hom_{c(a,t)}(x(a,t),x(a,t))$ coming from the action of $d(a,t)$ on $x(a,t)$. 
Hence $\gamma$ is the composition of two quasi-isomorphisms and therefore a quasi-isomorphism itself.
This completes the proof of the theorem.\qed

\ssection{4. The operators $\mathbb{O}_p$.}

Recall $\mathbb{O}_p$ is the endo-2-functor of $\mathcal{T}$ defined by 
$$\mathbb{O}_p(A,T) = (c_p,x_p) \ast \bA = (c_p(\bA), x_p(\bA)).$$
Here we compare the operator $\mathbb{O}_p$ with the operator $\mathcal{C}_p$, 
and consider braid group actions relative to these operators.

Quasi-hereditary algebras were invented by Cline, Parshall, and Scott \cite{CPS}, 
whilst their tilting theory was 
developed by Ringel.
We refer to Donkin's book for an account of the general theory \cite{Donkin}.

Let $\mathcal{U}$ be the sub-$2$-category of $\mathcal{T}$, consisting of pairs $(A,T)$,
where $A$ is a quasi-hereditary algebra and $T$ is a self-dual tilting bimodule.

\begin{prop} \label{COcompare}
The weak $2$-category $\mathcal{U}$ is stable under $\mathcal{C}_p$ and $\mathbb{O}_p$.
We have $\mathbb{O}_p(\bA) \cong \mathcal{C}_p(\bA)$, for $\bA \in \mathcal{U}$.
\end{prop}
\begin{proof}
In previous work, we established that $\mathcal{U}$ is stable under $\mathcal{C}_p$ (\cite{MT}, Corollary 20).
It therefore suffices to show that, 
whenever $A$ is a quasi-hereditary algebra with a self-dual tilting bimodule $T$,  
$$(c_p(A,T), x_p(A,T)) \cong (\cC_p(A), \cX_p(A)).$$
To see this, we identify the sum of copies of $A \subset \cC_p(A)$ with $c_p^{0} \otimes A \subset c_p(A,T)$;
we identify the sum of copies of $T \subset \cC_p(A)$ with $c_p^{(1,0)} \otimes T \subset c(A,T)$;
we identify the sum of copies of $T^* \subset \cC_p(A)$ with $c_p^{(0,1)} \otimes T \subset c(A,T)$  by self-duality; 
we identify the sum of copies of  $A^* \subset \cC_p(A)$ with $c_p^{(1,1)} \otimes T^2 \subset c(A,T)$,
via the dual of the isomorphism 
$$T^{2*} = \Hom_F(T \otimes_A T, F) \cong \Hom_A(T,\Hom(T,F)) \cong \Hom_A(T,T) \cong A.$$ 
Note the isomorphism $\Hom_A(T,T) \cong A$ holds since $T$ is a tilting bimodule for $A$.
By functorality, the resulting isomorphism of vector spaces $\cC_p(A) \cong c(A,T)$ 
is an algebra isomorphism.
Similarly, we define an isomorphism of bimodules $\cX_p(A) \cong x(A,T)$,
identifying the sum of copies of $A \subset \cX_p(A)$ with $x_p^{(0,0)} \otimes A \subset x(A,T)$;
the sum of copies of $T \subset \cX_p(A)$ with $x_p^{(1,0)} \otimes T \subset x(A,T)$;
the sum of copies of $T^* \subset \cX_p(A)$ with $x_p^{(0,1)} \otimes T \subset x(A,T)$; 
and the sum of copies of  $A^* \subset \cX_p(A)$ with $x_p^{(1,1)} \otimes T^2 \subset x(A,T)$. 
\end{proof}

Since the pair $(F,F)$ consists of a Ringel self-dual algebra with a self-dual tilting bimodule,
upon writing $E_n$ for the algebra component $E_n$ of $\mathbb{O}_{p}^n(F,F)$,
we obtain the following corollary:

\begin{cor} \label{CEiso}
The algebra $C_n$ is isomorphic to $E_n$.
\end{cor}

\begin{lem} \label{dgres}
Suppose that $(A,T) \in \mathcal{U}$, and $t \twoheadrightarrow T$ 
is a projective bimodule resolution of $T$.
Then the natural algebra homomorphism $c_p(A,t) \twoheadrightarrow c_p(A,T)$ is a quasi-isomorphism.
\end{lem}
\proof  Since all modules appearing in $t$ are projective as bimodules, 
tensoring with $t$ is exact. 
Hence $t \otimes_A t$ is a resolution of $T \otimes_A t$. 
But again since all modules appearing in $t$ are projective as bimodules and hence filtered by standard modules, 
tensoring with $T$ over $A$ is exact on $t$ (see \cite{Donkin}), 
hence $t \otimes_A t $ is a resolution of $T \otimes_A T$.
Since $c_p$ is concentrated in degrees $0,1$, and $2$,
we deduce the homology of $c_p(A, t)$ is indeed isomorphic to $c_p(A,T)$. 
\endproof

Let
$$Y_i:\cdots 0 \rightarrow c_p e_i \otimes e_i c_p \rightarrow c_p \rightarrow 0 \cdots,$$ 
where $e_i$ is the primitive idempotent corresponding to one of the vertices $1, \dots, p-1$ in the quiver of $c_p$. 
By a theorem of Khovanov-Seidel 
(\cite{KS}, Proposition 2.4) and Rouquier-Zimmermann \cite{RZ}, 
tensoring with this complex induces a self-equivalence of $D^b(c_p \ml)$.

\begin{thm} \label{braidgroup}
Let $(a,t) \in \mathcal{T}$, where $a$ is an associative algebra, and $_at_a$ is a 
Rickard tilting complex of bimodules. Let $1 \leq i \leq p-1$.
There is a derived self-equivalence 
$$s_i: D_{dg}(c_p(a,t) ) \overset{Y_i(a,t) \otimes^L -}{\longrightarrow} D_{dg}(c_p(a,t)).$$ 
The derived equivalences $s_i$ satisfy braid relations 
$s_is_{i+1}s_i \simeq s_{i+1}s_is_{i+1}$ and $s_is_j \simeq s_js_i$, for $|i-j| \geq 1$.
\end{thm}
\proof
We apply the theory of previous section $3$ to $c_p$. 
The role of the algebra $c$ and $d$ is played by the algebra $c_p$ 
and that of the complex $x$ by the Rickard tilting complex $Y_i$.
It follows from Theorem \ref{derivedequivalence}, that we have a derived equivalence $s_i$ as stated.

We now prove the braid relations.
Define $Y_i':= 0 \rightarrow c_p \rightarrow c_pe_i \otimes e_i c_p \langle -1 \rangle \rightarrow 0$, 
where $c_p$ is in homological degree zero. 
Then $Y_i \otimes_{c_p} Y_i'$ and $Y_i' \otimes_{c_p} Y_i$ are both quasi-isomorphic to $c_p$, 
as graded complexes (see \cite{KS}, Proposition 2.4). 

Furthermore, $Y_i \otimes_{c_p} Y_{i+1} \otimes_{c_p} Y_i'$ is quasi-isomorphic to 
$Y_{i+1}' \otimes_{c_p} Y_i \otimes_{c_p} Y_{i+1}$, 
again as graded complexes (see the proof \cite{KS}, Theorem 2.5). 
From this it follows that $Y_i \otimes_{c_p} Y_{i+1} \otimes_{c_p} Y_i$ 
is quasi-isomorphic to $Y_{i+1} \otimes_{c_p} Y_i \otimes_{c_p} Y_{i+1}$ 
(by tensoring both complexes on the right with $Y_i$ and on the left with $Y_{i+1}$). 
Now note that, in the setup of Definition \ref{x(a,t)}, 
it is clear that for two bimodules $x$ and $y$, 
$x(a,t) \otimes_{c(a,t)} y(a,t)$ is isomorphic to $(x \otimes_c y)(a,t)$.
Therefore $s_is_{i+1}s_i$ is represented by $(Y_i \otimes Y_{i+1} \otimes Y_i)(a,t)$, 
and $s_{i+1}s_is_{i+1}$ is represented similarly. Now Lemma \ref{preserveqim} implies the Theorem.
\endproof

\begin{cor} \label{braidonalgebra}
If $(A,T) \in \mathcal{U}$, then the derived category $D^b(c_p(A,T) \ml)$ admits a weak braid group action.
\end{cor}
\proof This follows from Theorem \ref{braidgroup}. 
Let $t$ denote a projective bimodule resolution of $_A T_A$. Since $A$ is quasi-hereditary, 
it has finite global dimension \cite{CPS}.
Therefore, since $T$ is a tilting module, $t$ is a Rickard tilting complex.
We have quasi-isomorphism $c_p(A, t) \rightarrow c_p(A, T)$ by Lemma \ref{dgres}, 
which induces an equivalence between $D(c_p(A,T) \ml)$ and $D_{dg}(c_p(A,T))$ 
(see \cite{Ke2}, Proposition 6.2).  
The self-equivalence $s_i$ therefore defines a self-equivalence $s_i$ of $D(c_p(A,T) \ml)$, 
under which compact objects map to compact objects.
Since compact objects in $D(c_p(A,T) \ml)$ can be identified with $c_p(A,T) \perf$, 
by an observation of Rickard,
which can in turn be identified with $D^b(c_p(A,T) \ml)$ because $c_p(A,T)$ has finite global dimension, 
the self-equivalence $s_i$ restricts to a self-equivalence of $D^b(c_p(A,T) \ml)$.
By Theorem \ref{braidgroup}, these equivalences satisfy braid relations, defining a weak braid group
action on $D^b(c_p(A,T) \ml)$.
\endproof

\begin{example}
As an example, let us describe the dg-algebra $c_2(a,t)$ where $a \cong c_2$, and $t$ is a resolution of $x_2$. 
The algebra $c_2$ is a $5$-dimensional algebra with basis 
$\{e_1,e_2, \xi, \eta, \xi\eta\}$, where the underlying quiver is
$$
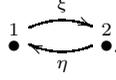
\xymatrix{\overset{1}{\bullet}\ar@/^/[r]^{\xi} &\overset{2}{\bullet}\ar@/^/[l]^{\eta}.}$$ 
The idempotents are in degree $0$, $\xi$ and $\eta$ are in degree $1$ and $\xi\eta$ is therefore in degree $2$. 
The tilting module of $c_2$ is given by $c_2 e_1 \oplus c_2 e_1/\Rad c_2e_1 $. 
This is isomorphic to the cone of a complex 
$$t: \cdots 0 \rightarrow c_2 \rightarrow c_2e_1 \otimes e_1 c_2 \rightarrow 0...,$$ 
which is obviously projective on both sides.
To obtain $c_2(c_2, t)$ from $c_2$, 
we replace the idempotents $e_1$ and $e_2$ with copies of $c_2$, 
we replace $\xi$ and $\eta$ with copies of the complex $t$ and 
$\xi\eta$ with $t \otimes_{c_2} t$.
A dg bimodule representing the self-equivalence $s_1$ of $D_{dg}(c_2(c_2,t))$
is given by the total complex of the complex of complexes
$$... 0  \rightarrow c_2(c_2,t) (e_1 \otimes 1) \otimes_{c_2} (e_1 \otimes 1).c_2(c_2,t) \rightarrow  c_2(c_2,t) \rightarrow 0 ....$$
\end{example}

\ssection{5. Applications to the representation theory of $GL_2$.}

We gather here a number of results concerning the representations of $GL_2$, which follow from
our theory.

Recall the category of polynomial representations of $GL_2$ is equivalent to the category of finite dimensional
modules over the Schur algebra $S(2)$ \cite{Green}. 

A $\mathbb{Z}_+$-grading is said to be \emph{tight} if its degree zero component is semisimple.

\begin{thm}
The Schur algebra $S(2)$ admits a tight $\mathbb{Z}_+$-grading.
\end{thm}
\proof
Every block of the Schur algebra $S(2)$ is Morita equivalent to the quotient of a block $B_n$ by an 
ideal $B_n i B_n$, for some $n \geq 0$ and some idempotent $i$.
Since $A_n$ is tightly graded by definition, $B_n$ is tightly graded by Theorem \ref{allfour}.
The theorem follows. 
\endproof

Recall $E_n$ was defined to be the algebra component of $\mathbb{O}_p^n(F,F)$.
As we have a natural algebra homomorphism from $c_p$ to $F$, which sends vertices $2,..,p$ to zero,
and vertex $1$ to $1_F$, we have an inherited algebra homomorphism
from the algebra component of $\mathbb{O}_p(A,T)$ to $A$, for any $(A,T) \in \mathcal{T}$.
We thus have a directed sequence of algebra homomorphisms.
$$E_0 \leftarrow E_1 \leftarrow E_2 \leftarrow ...$$
The operator $\mathbb{O}_p$ controls the rational representation theory of $GL_2$:

\begin{thm}
Every block of rational representations of $GL_2$ is equivalent to the category of
$\lim_n E_n$-modules. 
\end{thm}
\proof
As we noted in a previous article, every block of representations of $GL_2$ is equivalent to the category of
$\lim_n B_n$-modules \cite{MT}. 
By Theorem \ref{allfour} and Corollary \ref{CEiso}, we have $B_n \cong E_n$. The theorem follows.
\endproof

\begin{thm}
The derived category of every block of the Schur algebra with $p^r$ irreducible modules 
admits a weak action of the braid group $Br_{p-1}$.
\end{thm}
\proof
Every block of the Schur algebra with $p^r$ irreducible modules is Morita equivalent to $B_n$,
hence to $E_n$ by Theorem \ref{allfour}. 
But $E_n$ is the algebra component $c_p(A,T)$ of $\mathbb{O}_p(A,T)$, 
for some $(A,T) \in \mathcal{U}$, by Proposition \ref{COcompare}.
Therefore, the derived category of $E_n$ admits a weak braid group action, by Corollary \ref{braidonalgebra} 
\endproof

\bibliographystyle{amsplain}
\bibliography{qh2}

\providecommand{\bysame}{\leavevmode\hbox to3em{\hrulefill}\thinspace}
\providecommand{\MR}{\relax\ifhmode\unskip\space\fi MR }
\providecommand{\MRhref}[2]{%
  \href{http://www.ams.org/mathscinet-getitem?mr=#1}{#2}
}
\providecommand{\href}[2]{#2}
\begin{thebibliography}{10}

\bibitem{BaezDolan}
John~C. Baez and James Dolan, \emph{Higher-dimensional algebra and topological
  quantum field theory}, J. Math. Phys. \textbf{36} (1995), no.~11, 6073--6105.
  \MR{MR1355899 (97f:18003)}

\bibitem{BFK}
Joseph Bernstein, Igor Frenkel, and Mikhail Khovanov, \emph{A categorification
  of the {T}emperley-{L}ieb algebra and {S}chur quotients of
  {$U(\mathfrak{sl}\sb 2)$} via projective and {Z}uckerman functors}, Selecta
  Math. (N.S.) \textbf{5} (1999), no.~2, 199--241. \MR{MR1714141 (2000i:17009)}

\bibitem{Broue}
Michel Brou{\'e}, \emph{Isom\'etries parfaites, types de blocs, cat\'egories
  d\'eriv\'ees}, Ast\'erisque (1990), no.~181-182, 61--92. \MR{MR1051243
  (91i:20006)}

\bibitem{BK}
J.-L. Brylinski and M.~Kashiwara, \emph{Kazhdan-{L}usztig conjecture and
  holonomic systems}, Invent. Math. \textbf{64} (1981), no.~3, 387--410.
  \MR{MR632980 (83e:22020)}

\bibitem{Cheng}
Eugenia Cheng and Aaron Lauda, \emph{Higher-dimensional categories: an
  illustrated guide book}, available at http://www.cheng.staff.shef.ac.uk/.

\bibitem{ChuangRouquier}
Joseph Chuang and Rapha{\"e}l Rouquier, \emph{Derived equivalences for
  symmetric groups and {$\mathfrak {sl}\sb 2$}-categorification}, Ann. of Math.
  (2) \textbf{167} (2008), no.~1, 245--298. \MR{MR2373155}

\bibitem{CPS}
E.~Cline, B.~Parshall, and L.~Scott, \emph{Finite-dimensional algebras and
  highest weight categories}, J. Reine Angew. Math. \textbf{391} (1988),
  85--99. \MR{MR961165 (90d:18005)}

\bibitem{De}
Maud Devisscher, \emph{Some problems in the representation theory of reductive
  groups and their finite subgroups}, Ph.D. thesis, University of Oxford, 2003.

\bibitem{Donkin}
S.~Donkin, \emph{The {$q$}-{S}chur algebra}, London Mathematical Society
  Lecture Note Series, vol. 253, Cambridge University Press, Cambridge, 1998.
  \MR{MR1707336 (2001h:20072)}

\bibitem{EH}
Karin Erdmann and Anne Henke, \emph{On {R}ingel duality for {S}chur algebras},
  Math. Proc. Cambridge Philos. Soc. \textbf{132} (2002), no.~1, 97--116.
  \MR{MR1866327 (2002j:20081)}

\bibitem{Green}
J.~A. Green, \emph{Polynomial representations of {${\rm GL}\sb {n}$}},
  augmented ed., Lecture Notes in Mathematics, vol. 830, Springer, Berlin,
  2007, With an appendix on Schensted correspondence and Littelmann paths by K.
  Erdmann, Green and M. Schocker. \MR{MR2349209}

\bibitem{He2}
Anne~E. Henke, \emph{Schur subalgebras and an application to the symmetric
  group}, J. Algebra \textbf{233} (2000), no.~1, 342--362. \MR{MR1793600
  (2001m:20063)}

\bibitem{Ke2}
Bernhard Keller, \emph{Introduction to {$A$}-infinity algebras and modules},
  Homology Homotopy Appl. \textbf{3} (2001), no.~1, 1--35 (electronic).
  \MR{MR1854636 (2004a:18008a)}

\bibitem{Ke}
\bysame, \emph{On differential graded categories}, International Congress of
  Mathematicians. Vol. II, Eur. Math. Soc., Z\"urich, 2006, pp.~151--190.
  \MR{MR2275593}

\bibitem{KS}
Mikhail Khovanov and Paul Seidel, \emph{Quivers, {F}loer cohomology, and braid
  group actions}, J. Amer. Math. Soc. \textbf{15} (2002), no.~1, 203--271
  (electronic). \MR{MR1862802 (2003d:53155)}

\bibitem{Ko2}
Hitoshi Koshita, \emph{On quiver and relations for {${\rm SL}(2,2\sp n)$} in
  characteristic {$2$}}, S\=urikaisekikenky\=usho K\=oky\=uroku (1994),
  no.~877, 46--49, Representation theory of finite groups and algebras
  (Japanese) (Kyoto, 1993). \MR{MR1332086}

\bibitem{Kop}
\bysame, \emph{Quiver and relations for {${\rm SL}(2,2\sp n)$} in
  characteristic {$2$}}, J. Pure Appl. Algebra \textbf{97} (1994), no.~3,
  313--324. \MR{MR1314582 (96g:16035)}

\bibitem{Leinster}
Tom Leinster, \emph{Basic bicategories}, available at
  http://www.maths.gla.ac.uk/~tl/.

\bibitem{MT}
Vanessa Miemietz and Will Turner, \emph{Rational representations of
  $\rm{GL}_2$}, preprint.

\bibitem{Nebe}
Gabriele Nebe, \emph{The group ring of {${\rm SL}\sb 2(2\sp f)$} over 2-adic
  integers}, J. Reine Angew. Math. \textbf{528} (2000), 183--200. \MR{MR1801661
  (2001i:20029)}

\bibitem{Nebe2}
\bysame, \emph{The group ring of {${\rm SL}\sb 2(p\sp f)$} over {$p$}-adic
  integers for {$p$} odd}, J. Algebra \textbf{230} (2000), no.~2, 424--454.
  \MR{MR1775798 (2001i:20008)}

\bibitem{Rickard}
Jeremy Rickard, \emph{Derived equivalences as derived functors}, J. London
  Math. Soc. (2) \textbf{43} (1991), no.~1, 37--48. \MR{MR1099084 (92b:16043)}

\bibitem{RZ}
Rapha{\"e}l Rouquier and Alexander Zimmermann, \emph{Picard groups for derived
  module categories}, Proc. London Math. Soc. (3) \textbf{87} (2003), no.~1,
  197--225. \MR{MR1978574 (2004h:16003)}

\bibitem{Stroppel}
Catharina Stroppel, \emph{Categorification of the {T}emperley-{L}ieb category,
  tangles, and cobordisms via projective functors}, Duke Math. J. \textbf{126}
  (2005), no.~3, 547--596. \MR{MR2120117 (2005i:17011)}

\end{thebibliography}

\end{document}